\documentclass[11pt, a4paper]{amsart}

\usepackage{amsfonts}
\usepackage{amsmath}
\usepackage{amsthm}
\usepackage{amssymb}
\usepackage{pb-diagram}
\usepackage{centernot}
\usepackage{mathtools}
\usepackage{url}
\usepackage{stmaryrd}
\usepackage{hyperref}
\usepackage[capitalize]{cleveref}

\usepackage{mathrsfs}

\newtheorem*{thm*}{Theorem}
\newtheorem*{cor*}{Corollary}
\newtheorem*{prop*}{Proposition}
\newtheorem{theorem}{Theorem}[section]
\newtheorem{thm}[theorem]{Theorem}
\newtheorem{prop}[theorem]{Proposition}

\newtheorem{fact}[theorem]{Fact}
\newtheorem{cor}[theorem]{Corollary}

\newtheorem{lemma}[theorem]{Lemma}

\newtheorem*{mainconj}{Main Conjecture}

\newtheorem{claim}{Claim}
\newtheorem*{claim*}{Claim}

\theoremstyle{definition}
\newtheorem{defn}[theorem]{Definition}
\newtheorem{defin}[theorem]{Definition}
\newtheorem{definition}[theorem]{Definition}
\newtheorem{example}[theorem]{Example}

\theoremstyle{remark}
\newtheorem{rmk}[theorem]{Remark}
\newtheorem{remark}[theorem]{Remark}

\crefname{fact}{fact}{facts}
\Crefname{fact}{Fact}{Facts}
\Crefname{claim}{Claim}{Claims}
\Crefname{thm}{Theorem}{Theorems}

\newcommand{\acvf}{\ensuremath{\mathsf{ACVF}}}
\newcommand{\rcf}{\ensuremath{\mathsf{RCF}}}
\newcommand{\doag}{\ensuremath{\mathsf{DOAG}}}

 \DeclareMathOperator{\longpart}{Long}
 \DeclareMathOperator{\ind}{ind}

\newcommand{\la}{\langle}
\newcommand{\ra}{\rangle}

\newcommand{\sub}{\subseteq}

\newcommand{\es}{\ensuremath{\emptyset}}

\newcommand{\ldim}{\ensuremath{\textup{ldim}}}

\newcommand{\bb}[1]{\ensuremath{\mathbb{#1}}}

\newcommand{\cal}[1]{\ensuremath{\mathcal{#1}}}

\newcommand{\Lrarr}{\ensuremath{\Leftrightarrow}}
\newcommand{\Rarr}{\ensuremath{\Rightarrow}}

\newcommand{\res}{\ensuremath{\upharpoonright}}

\newcommand{\lam}{\ensuremath{\lambda}}

\newcommand{\Lam}{\ensuremath{\Lambda}}

\newcommand{\dom}{\ensuremath{\textup{dom}}}

\newcommand{\Aut}{\ensuremath{\textup{Aut}}}

\newcommand{\sm}{\setminus}

\newcommand{\dcl}{\operatorname{dcl}}

\newcommand{\scl}{\ensuremath{\operatorname{scl}}}

\DeclareMathOperator{\rk}{scl-rk}

\newcommand{\tp}{\mathrm{tp}}

\newcommand{\wort}{\mathrel{\perp^{\!\!\mathrm{w}}}}

\newcommand{\nwort}{\mathrel{{\centernot\perp}^{\!\!\mathrm{w}}}}

\newcommand{\proves}{\vdash}

\newcommand{\invtypes}{S^{\mathrm{inv}}}

\newcommand{\monster}{\mathfrak U}

\newcommand{\monsterbis}{\mathfrak V}

\newcommand{\bla}[4]{{#1}_{#2}#3\ldots#3{#1}_{#4}}

\newcommand{\domeq}{\mathrel{\sim_\mathrm{D}}}

\newcommand{\doms}{\mathrel{\ge_\mathrm{D}}}

\newcommand{\restr}{\upharpoonright}

\newcommand{\invext}{\mid}

\newcommand{\invtilde}{\operatorname{\widetilde{Inv}}(\monster)}

\newcommand{\invtildeof}[1]{\operatorname{\widetilde{Inv}}({#1})}

\newcommand{\deftilde}{\operatorname{\widetilde{Def}}(\monster)}

\newcommand{\into}{\hookrightarrow}

\DeclarePairedDelimiter{\abs}{\lvert}{\rvert}
\DeclarePairedDelimiter{\set}{\{}{\}}
\DeclarePairedDelimiter{\class}{\llbracket}{\rrbracket}

\usepackage{xcolor} 
\usepackage[normalem]{ulem} 

\newcommand{\deftypes}{S^{\mathrm{def}}}

\renewcommand{\phi}{\varphi}

\let\oldqed\qedsymbol
\newcommand{\qedclaim}{\mbox{$\underset{\textsc{claim}}{\square}$}}
\makeatletter
\newenvironment{claimproof}[1][\it Proof of Claim]{
\let\qedsymbol\qedclaim
  \par
  \pushQED{\qed}%
  \normalfont \topsep6\p@\@plus6\p@\relax
  \trivlist
\item[\hskip\labelsep
  \upshape
  #1\@addpunct{.}]\ignorespaces
}{%
  \popQED\endtrivlist\@endpefalse
}
\makeatother
\let\qedsymbol\oldqed

\title[Orthogonality and domination in o-minimal structures]{Orthogonality and domination in o-minimal expansions of ordered groups}

\author[P. And\'ujar Guerrero]{Pablo And\'ujar Guerrero}
\address{Pablo And\'ujar Guerrero, Facultad de ciencias matemáticas, Universitat de València, Burjassot 46100, Spain}
\email{pablo.andujar@uv.es}

\author[P.E. Eleftheriou] {Pantelis  E. Eleftheriou}
\address{Pantelis  E. Eleftheriou, School of Mathematics, University of Leeds, Leeds LS2 9JT, United Kingdom}
\email{P.Eleftheriou@leeds.ac.uk}

\author[R. Mennuni]{Rosario Mennuni}
\address{Rosario Mennuni, Dipartimento di Matematica, Università di Milano, Via Saldini 50, 20133 Milano, Italy}
\email{R.Mennuni@posteo.net}

\thanks{The first and second authors were partially supported by an EPSRC Early Career Fellowship (EP/V003291/1). RM was supported by the German Research Foundation (DFG) via HI 2004/1-1 (part of the French-German ANR-DFG project GeoMod) and under Germany’s Excellence Strategy EXC 2044-390685587, `Mathematics Münster: Dynamics-Geometry-Structure', by the projects PRIN 2017: ``Mathematical Logic: models, sets, computability'' Prot.~2017NWTM8RPRIN, PRIN 2022: ``Models, sets and classifications'' Prot.~2022TECZJA, and PRIN 2022 ``Logical methods in combinatorics'', 2022BXH4R5, Italian Ministry of University and Research. He acknowledges the MIUR Excellence Department Project awarded to the Department of Mathematics, University of Pisa, CUP I57G22000700001, and is a member of the INdAM research group GNSAGA}

\date{\today}

\begin{document}

\begin{abstract}
We analyse domination between invariant types in o-minimal expansions of ordered groups, showing that the domination poset decomposes as the direct product of two posets: the domination poset of an o-minimal expansion of a real closed field, and one derived from 
a linear o-minimal structure. We prove that if the Morley product is well-defined on the former poset, then the same holds for the  poset computed in the whole structure. We establish our results by employing the `short closure' pregeometry (\scl) in semi-bounded o-minimal structures, showing that types of \scl-independent tuples are weakly orthogonal to types of short tuples.  As an application we prove that, in an o-minimal expansion of an ordered group, every definable type is domination-equivalent to a product of $1$-types. Furthermore, there are precisely two or four classes of definable types up to domination-equivalence, depending on whether a global field is definable or not.
\end{abstract}

\maketitle
\section{Introduction}
\subsection{Overview}
The \emph{domination poset} $\invtilde$ of a saturated first-order structure $\monster=(U,\ldots)$ is an object encoding the relations holding between \emph{invariant types} up to small information.  It was first computed in the theory of algebraically closed valued fields (\acvf) in~\cite{hhm}, developed further in~\cite{invbartheory}, and has recently been studied in various contexts. If domination is compatible with the \emph{Morley product} $\otimes$, then $\invtilde$ has the structure of a partially ordered monoid, known as the \emph{domination monoid}.

A common theme has been to prove a decomposition theorem that enables the understanding of invariant types over $U$ in terms of simpler ones: for instance, invariant types in a simpler, stably embedded structure. For example, in the cases of \acvf{} and of real closed valued fields ($\mathsf{RCVF}$), it was shown that $\invtilde\cong\invtildeof{k(\monster)}\times \invtildeof{\operatorname{\Gamma}(\monster)}$, where $k$ denotes the residue field and $\Gamma$ the value group (see~\cite{hhm,ehm,dominomin}). In other words, every invariant type in these contexts is domination-equivalent to the product of an invariant type in the residue field with one in the value group. Moreover, every invariant type is domination-equivalent to a Morley product of $1$-types. Similar but more involved results hold for larger classes of Henselian valued fields. This was proven in~\cite{hilsDominationMonoidHenselian2024}, where it was also shown that the monoid structure on $\invtilde$ transfers from those of the residue field and value group, when available.

Parallel to these works, the third author~\cite{dominomin} analysed $\invtilde$ in the o-minimal setting, proving it is a monoid, and obtaining a structure theorem (\Cref{fact:redto1tps} below), under the assumption that every invariant type is domination-equivalent to a product of $1$-types. He also showed that the latter property holds in the theory \rcf{} of real closed fields, in the theory \doag{} of divisible ordered abelian groups, and in o-minimal theories with no non-simple types in the sense of~\cite{mayer_vaughts_1988}. Conjecturally, this property holds in every o-minimal structure.
\begin{mainconj}
  For every o-minimal theory $T$ and monster model $\monster\models T$, every type over $U$ that is invariant over a small set is domination-equivalent to a product of $1$-types.
\end{mainconj}

Although a proof of the Main Conjecture in arbitrary o-minimal theories remains rather elusive and open,  in this paper we show that knowing it for o-minimal expansions of \rcf{} is enough to obtain it for o-minimal expansions of \doag. This reduction follows from our main theorem (\Cref{thm:main1}), a decomposition result for the domination poset of non-linear o-minimal expansions of ordered groups that do not define a global field. The main technical lemma in its proof (\Cref{lem:rfsbteorcf}) also yields  the Main Conjecture for o-minimal theories that are intermediate between \doag{} and \rcf{} (\Cref{co:RpureRCF}). At the same time,  we prove the Main Conjecture for linear o-minimal expansions of $\mathsf{DOAG}$ (\Cref{lineardec}).
 
We achieve \Cref{thm:main1}  by employing a fine machinery developed for \emph{semi-bounded} o-minimal structures by the second author in \cite{el-sbdgroups}. These structures were first extensively studied in the 90s by several authors~\cite{betgrrng, mpp, pet-reals}, they relate to Zilber's dichotomy principle on definable groups and fields, and have been developed further in~\cite{ed1, pet-sbd, el-sbdgroups}. Examples include the real ordered group expanded by multiplication restricted to a bounded interval.  Statements known for the two extreme cases, pure group and full field structure, appear to be particularly hard to prove in the semi-bounded case. See, for example, the analysis of definable groups in~\cite{elp-sel2} in this setting. In the introduction of that paper, the authors alluded to a possible orthogonality between certain sets, or types, that are very closely related to the two extreme structures. In this paper, we formalise this intuition by decomposing invariant types into \emph{orthogonal} `long' and `short' parts. We then use this orthogonality to prove the aforementioned decomposition theorem for $\invtilde$ (\Cref{thm:main1}).

While the above results hold at the level of invariant types, we are able to prove the Main Conjecture for the subclass of \emph{definable types} in o-minimal expansions of \doag, by combining the aforementioned decomposition with results of the first author, Thomas and Walsberg~\cite{guerrero_compactness, andujarguerreroDirectedSetsTopological2021} on o-minimal types and definability of cofinal curves for definable directed preorders. In fact, in \Cref{thm:deftildedoag}, we fully compute the submonoid $\deftilde$ of $\invtilde$ arising from definable types.

\subsection{Invariant types and domination monoids}
We now recall briefly some terminology that will allow us to state precisely our results. For more detailed definitions, we refer the reader to  \Cref{sec:domination}.

Let $T$ be a theory and $\monster=(U,\ldots)$ a monster model.
A \emph{global type}, that is, a type over $U$, is \emph{invariant} if it is fixed under the natural action of $\Aut(\monster/A)$ for some small $A$. A special case is that of definable types, studied in the o-minimal context in~\cite{marker_definable_1994}.

 A natural relation on invariant types is that of \emph{domination}: we say that $p(x)$ \emph{dominates} $q(y)$ if  $p(x)$ entails $q(y)$ modulo a  type $r(x,y)$ over a small set, and call $p$ and $q$ \emph{domination-equivalent} if they dominate each other (\Cref{dfn:domination}). The corresponding quotient $\invtilde$ of the space $\invtypes_{<\omega}(U)$ of invariant types in finitely many variables (simply denoted by  $\invtypes(U)$) is the \emph{domination poset}; its definition is implicit in~\cite{hhm},  its general theory was developed in~\cite{invbartheory}, and a version for measures was recently defined in~\cite{gy}.

The space $\invtypes(U)$ naturally carries a monoid operation $\otimes$. In certain cases (but not always~\cite{invbartheory}), this operation is compatible with domination, hence descends to $\invtilde$, equipping it with the structure of a partially ordered monoid, the
\emph{domination monoid}.  

Besides the results on valued fields already recalled above, computations of $\invtilde$ are available for regular ordered abelian groups and short exact sequences of abelian structures~\cite{hilsDominationMonoidHenselian2024}, certain expansions of dense-meet trees~\cite{mennuni_weakly_2022}, and more. In the o-minimal setting it was shown~\cite{dominomin} that, in every theory where the Main Conjecture holds, $\invtilde$ is a monoid, and may be identified with the upper semilattice $(\mathscr P_{<\omega}(X(\monster)), \cup, \supseteq)$ of finite subsets of the set $X(\monster)$ of classes of non-realised invariant $1$-types modulo definable bijection (\Cref{fact:redto1tps}).

\subsection{Semi-bounded o-minimal structures}
Recall that, for an o-minimal expansion $\cal{M}=\langle M, <, +, 0, \ldots\rangle$ of an ordered group, there are naturally three possibilities: $\cal{M}$ is either (a) linear, (b) semi-bounded (and non-linear), or (c) it expands a real closed field. Let us define the first two.

\begin{defn}
Let $\Lambda$ be the set of all $\es$-definable partial endomorphisms 
of $\langle M, <, +, 0\rangle$, and $\mathcal{B}$ the collection of all bounded definable sets. Then $\mathcal{M}$ is called \emph{linear} (\cite{lp}) if every definable set is already definable in $\langle M, <, +, 0, \{\lam\}_{\lam\in \Lambda}\rangle$, and it is called \emph{semi-bounded} (\cite{pet-reals,ed1}) if every definable set is already definable in $\langle M, <, +, 0, \{\lam\}_{\lam\in \Lambda}, \{B\}_{B\in \mathcal{B}}\rangle$.
\end{defn}

Obviously, if $\mathcal{M}$ is linear then it is semi-bounded. By \cite{pest-tri}, $\cal{M}$ is not linear if and only if there is a definable real closed field. By \cite{ed1}, $\cal{M}$ is not semi-bounded if and only if  $\cal M$ expands a real closed field, if and only if for any two open intervals there is a definable bijection between them, if and only if there is \emph{pole}: a definable bijection between a bounded interval and an unbounded one. An important example of a semi-bounded non-linear structure is the expansion of the ordered vector space $\la \bb R, <,+, 0, x\mapsto \lam x\ra_{\lam\in \bb R}$ by all bounded semialgebraic sets.

\subsection{Results} We are now ready to state the results of this paper. In \Cref{lineardec}, we prove the Main Conjecture for linear o-minimal expansions of ordered groups, and in \Cref{co:RpureRCF} for o-minimal structures between \doag{} and \rcf{}. The proof of the latter corollary uses a technical statement, \Cref{lem:rfsbteorcf}, that we also employ in the proof of our main result, which is the following.

\begin{thm*}[\ref{thm:main1}]
  Let $T$ be an o-minimal semi-bounded theory, and $\monster$ a monster model. Assume there is an interval $R$ whose induced structure $\mathcal R$ expands a real closed field.
  Then there is a set $\operatorname{Long}(\monster)$ of domination-equivalence classes of $1$-types inducing an isomorphism of posets
  \[
    \invtilde\cong \invtildeof{\mathcal R(\monster)}\times (\mathscr P_{<\omega}(\operatorname{Long}(\monster)), \supseteq).
  \]
      If $\invtildeof{\mathcal R(\monster)}$ is a monoid, then so is $\invtilde$ and, if $\mathscr P_{<\omega}(\longpart(\monster))$ is equipped with the binary operation $\cup$, then the above is an isomorphism of monoids.
  Moreover,  $\mathscr P_{<\omega}(\operatorname{Long}(\monster))$ is a quotient of the domination monoid of the linear reduct of $\monster$.
\end{thm*}

The main tool in proving the above theorem concerns the interaction of full \scl-rank tuples with short ones (\Cref{prop-bcorth}), and can be phrased in terms of types as follows. (Here, $\scl(A)=\dcl(AR)$, see \Cref{sec:semibdd}).

 \begin{prop*} [\Cref{co:wort}] 
   Let $T$ be an o-minimal semi-bounded theory, and assume there is a definable interval $R$ whose induced structure expands a real closed field.
   For every parameter set $A$, if $p\in S(A)$ concentrates on a cartesian power of $R$, and $q\in S(A)$ is the type of an $\scl$-independent tuple, then $p$ and $q$ are weakly orthogonal (\Cref{def:worth}). 
  \end{prop*} 

In  \Cref{thm:deftildeexprcf}, we prove that the Main Conjecture holds in o-minimal expansions of \rcf{} for definable types, and fully compute the corresponding submonoid $\deftilde$.  In combination  with \Cref{thm:main1}, this yields a proof of the Main Conjecture for definable types, together with a computation of $\deftilde$, in every o-minimal expansion of \doag.

\begin{thm*}[\ref{thm:deftildedoag}]
Let $\monster$ be a monster model of an o-minimal expansion of $\mathsf{DOAG}$. Then $\deftilde$ is  a monoid, and isomorphic to 
  \begin{itemize}
  \item $(\mathscr P(\set{0,1}), \cup, \supseteq)$ if $T$ is semi-bounded;
  \item $(\mathscr P(\set{0}), \cup, \supseteq)$ otherwise.
  \end{itemize}
\end{thm*}

\subsection*{Structure of the paper.} In \Cref{sec:domination} we recall the necessary preliminaries on domination and prove the Main Conjecture in linear o-minimal theories, \Cref{lineardec}. \Cref{sec:cofcurvesdeftypes} deals with definable types in o-minimal expansions of \rcf, characterised up to domination-equivalence in \Cref{thm:deftildeexprcf}. We prove \Cref{co:wort} in \Cref{sec:semibdd} and, in our final \Cref{sec:mainresults}, we combine the results of the previous sections to prove \Cref{thm:main1,thm:deftildedoag}, and to prove the Main Conjecture for structures between \doag{} and \rcf{} (\Cref{co:RpureRCF}).

\subsection*{Acknowledgements.} We thank Y.~Peterzil for pointing out to us \Cref{fact-kobi} and \Cref{fn:kobi2}.

\section{Domination}\label{sec:domination}

In this section, we recall definitions and basic facts concerning orthogonality and domination.  We refer the reader to~\cite{invbartheory,dominomin,hilsDominationMonoidHenselian2024} for a more extensive treatment, as well as for proofs. 
We also prove the Main Conjecture for linear o-minimal expansions of ordered groups (\Cref{lineardec}).

\subsection{Orthogonality and domination}
Throughout the paper, we denote by $T$ a complete first-order theory, and by $\monster=(U,\ldots)$ a monster model, that is, a model that is $\kappa$-saturated and $\kappa$-strongly homogeneous for 
some fixed strong limit cardinal $\kappa=\kappa(\monster)> \abs T$. By \emph{small} we mean smaller than $\kappa$. For more details, see~\cite[Section~1.1]{dominomin}. We denote by $\monsterbis=(V,\ldots)$ a larger monster model, that is, an elementary extension $\monsterbis\succ \monster$ such that $U$ is small with respect to $\monsterbis$.  Parameter sets, denoted for example by $A$, are usually assumed to be small subsets of $U$. Parameter sets $B\supseteq U$ are assumed to be small subsets of $V$, and realisations of types over $U$ are assumed to be in $V$.

 From now on, except in \Cref{sec:semibdd} or if otherwise specified, `definable' means `definable with parameters from $U$'. On the contrary, if not otherwise specified, formulas and terms are assumed to be over $\emptyset$. The length of a tuple of variables $x$ is denoted by $\abs x$, and similarly for tuples of parameters; for these, we sometimes write e.g.\ $a\in U$ instead of $a\in U^{\abs a}$. When clear from context, $\abs\cdot$ will also be used to denote the absolute value function. 

We denote by $S_x(U)$ the space of types in the tuple of variables $x$ over $U$, by $S(U)$ the space of types over $U$ in any finite number of variables, and by $\invtypes(U)$ its subspace of invariant ones, that is, those fixed under the natural action of $\Aut(\monster/A)$ for some small $A$. Equivalently, $p(x)$ is $A$-invariant if, for every formula $\phi(x,y)$ over $\emptyset$, there is a set $d_p\phi\subseteq S_y(A)$ such that, for every $d\in U^{\abs y}$, we have $\phi(x,d)\in p(x)\iff \tp(d/A)\in d_p\phi$. If $p(x)$ is a type and $f(x)$  a definable function with domain a set on which $p$ concentrates, we write $f_*p$ for the pushforward of $p$ along $f$. Namely, if $p=\tp(a/A)$, then $f_*p=\tp(f(a)/A)$.

\begin{defin}\label{def:worth}
We say that two types $p(x),q(y)\in S(A)$ are \emph{weakly orthogonal}, and write $p\wort q$, if $p(x)\cup q(y)$ implies a complete type over $A$.
\end{defin}

\begin{definition} \label{dfn:domination} Let $p(x),q(y)\in \invtypes(U)$.
  \begin{enumerate}
  \item If $A\subseteq U$, we define $S_{pq}(A)$ to be the subspace of $S_{xy}(A)$ of types extending $(p(x)\cup q(y))\restr A$.
  \item  We say that $p$ \emph{dominates} $q$, and write $p\doms q$, if there are a small $A$ and $r\in S_{pq}(A)$ such that $p(x)\cup r(x,y)\proves q(y)$. We say that $r$ \emph{witnesses} that $p\doms q$.
  \item We say that $p$ and $q$ are \emph{domination-equivalent}, and write $p\domeq q$, if $p\doms q$ and $q\doms p$. The domination-equivalence class of $p$ is denoted by $\class p$.
  \item The \emph{domination poset} $(\invtilde, \doms)$ is the quotient $\invtypes(U)/\mathord{\domeq}$ together with the partial order induced by $\doms$.
  \end{enumerate}
\end{definition}
\begin{defin}\label{defin:invtensorth}Let $p,q\in \invtypes(U)$.
  \begin{enumerate}
  \item\label{point:definvext} If $B\supseteq U$, define $p\invext B$ as follows. Suppose $A$ is a small subset of $U$ and $p$ is $A$-invariant.  For $\phi(x,w)$ a formula over $\emptyset$ and $b\in B^{\abs w}$, let $\phi(x,b)\in p\invext B$ if for every/some $d\in U^{\abs b}$ with $d\equiv_A b$ we have $\phi(x,d)\in p$.
  \item Define $p(x)\otimes q(y)$ as follows. Fix $b\models q$. Then, for every $\phi(x,y)\in \mathcal L(U)$, set $\phi(x,y)\in p\otimes q$ if $\phi(x,b)\in p\invext U b$.
  \item We call $p$ and $q$ \emph{orthogonal}, and write $p\perp q$, if for every $B\supseteq U$ we have $(p\invext B)\wort (q\invext B)$. 
  \end{enumerate}
\end{defin}
We do not know whether in an o-minimal theory weak orthogonality of invariant types implies orthogonality, although in general this is false~\cite[Example~2.2]{hilsDominationMonoidHenselian2024}. The fact below is a summary of the known properties concerning the interaction of $\otimes$, domination, and (weak) orthogonality.
\begin{fact}\label{fact-wortdomeq}Let $p_0, p_1,q\in\invtypes(U)$.The following hold. 
  \begin{enumerate}
  \item\label{point:fwd1} If $p_0\doms p_1$ then $p_0\otimes q\doms p_1\otimes q$.
  \item\label{point:fwd2} If $p_0\doms p_1$ and for $i<2$ we have $p_i\otimes q\domeq q\otimes p_i$, then $q\otimes p_0\doms q\otimes p_1$.
  \item\label{point:fwd3}\label{point:wortdown} If $p_0\doms p_1$ and $p_0\wort q$, then $p_1\wort q$. In particular, if $p_0\doms p_1$ and $p_0\wort p_1$ then $p_1$ is realised.
  \item\label{point:fwd4} If $p_0\doms p_1$ and $p_0\perp q$, then $p_1\perp q$.
  \item\label{point:fwd5} If $p\wort q$, then $p$ and $q$ commute, that is,\footnote{See also~\cite[Definition~1.13]{dominomin} and the discussion surrounding it.} $p(x)\otimes q(y)=q(y)\otimes p(x)$.
  \item\label{point:fwd6} If $p_0\wort q_0$, $p_0\domeq p_1$ and $q_0\domeq q_1$, then $p_0\otimes q_0\domeq p_1\otimes q_1$.
  \item\label{point:fwd7} If for $i<2$ we have $p_i\perp q$ then $p_0\otimes p_1 \perp q$.
  \end{enumerate}
\end{fact}
\begin{proof}
  See~\cite[Lemma~1.14]{invbartheory} for~\ref{point:fwd1}, from which~\ref{point:fwd2} follows easily. See \cite[Proposition~3.13, Corollary~3.14]{invbartheory} for~\ref{point:fwd3}, \cite[Proposition~2.4]{hilsDominationMonoidHenselian2024} for~\ref{point:fwd4}, and \cite[Remark~1.15]{dominomin} for ~\ref{point:fwd5}. Applying the latter together with~\ref{point:fwd1}, \ref{point:fwd2} and~\ref{point:fwd3} yields~\ref{point:fwd6}, while~\ref{point:fwd7} is~\cite[Proposition~2.9]{hilsDominationMonoidHenselian2024}.
\end{proof}
\begin{definition}
  If for every $p_0,p_1,q\in\invtypes(U)$ we have $(p_0\doms p_1)\Longrightarrow (q\otimes p_0\doms q\otimes p_1)$, we say that $\otimes$ \emph{respects} $\doms$, or that $\otimes$ and $\doms$ are \emph{compatible}. In this case,  the expansion $(\invtilde, \doms, \otimes)$ of the domination poset by the operation induced by $\otimes$ is a monoid, which we call the \emph{domination monoid of $\monster$}. We abbreviate this fact by simply saying that \emph{$\invtilde$ is a monoid}.
\end{definition}

Recall that  $p(x)\in S(U)$ is \emph{definable} over $A$ if it is $A$-invariant and every $d_p\phi$ is clopen, equivalently, given by  a formula over $A$. We write $\deftypes(U)$ for the space of types over $U$ that are definable over some $A\subseteq U$.
\begin{fact}[{\cite[Theorem~3.5(2)]{invbartheory}}]\label{fact:domprdef}
If $p,q\in\invtypes(U)$, $p\doms q$, and $p$ is definable, then so is $q$.
\end{fact}
We write $\deftilde$ for the restriction of $\invtilde$ to domination-equivalence classes of definable types.

\begin{defn} Let $P$ be a $\emptyset$-definable unary set. Consider the language
  \[\mathcal L_{\ind}=\{R_{\phi}(x): \phi(x)\in \mathcal L\}.\]
  Given $\mathcal M\models T$, the \emph{induced structure on $P$ by $\mathcal M$}, denoted by $\mathcal P(\mathcal M)$, is the $\mathcal L_{\ind}$-structure whose universe is $P(\mathcal M)$ and, for every $\phi(x)\in \mathcal L$  and  $a\in P(\mathcal M)^{\abs x}$,
\[\mathcal P(\mathcal M)\models R_\phi(a) \iff \cal M\models \phi(a).\]
\end{defn}

\begin{defn}
    If $P$ is a $\emptyset$-definable unary set, we say that $P$ is \emph{stably embedded} if for every $\phi(x)\in \mathcal L(U)$, the set $\phi(\monster)\cap P(\monster)^{\abs x}$ is definable with parameters in the structure $\mathcal P(\monster)$.
  \end{defn}
  \begin{rmk}\*
    \begin{enumerate}
    \item One may still talk of stable embeddedness for $A$-definable unary $P$ by adding to the language constants for elements of $A$. Stable embeddedness of such a $P$ does not depend on the choice of $A$.
    \item Whether $P$ is stably embedded does not depend on the choice of monster $\monster$.
    \end{enumerate}
\end{rmk}

\begin{fact}\label{fact-fullemb}
Let $R$ be a stably embedded unary definable set.  Denote by $\iota: \invtypes_{\mathcal R^{<\omega}}(R(\monster))\to S(U)$  the map sending an invariant type of $\mathcal R(\monster)$ in finitely many variables to the unique type of $\monster$ it implies. Then every type in the image of $\iota$ is invariant, and $\iota$ induces an embedding of posets $\invtildeof{\mathcal R(\monster)}\into\invtilde$, which is also a homomorphism for the relations $\wort$, $\nwort$, $\perp$, $\centernot\perp$. If $\otimes$ respects $\doms$, then it is also an embedding of monoids.
\end{fact}
\begin{proof}
See~\cite[Proposition~2.3.31]{mennuniInvariantTypesModel2020a} and  \cite[Proposition~2.5]{hilsDominationMonoidHenselian2024}. Note that the latter uses the terminology \emph{fully embedded}. In the context of the present paper, whenever we consider definable sets as structures, we always mean with the full induced structure, in which case full embeddedness coincides with stable embeddedness.
\end{proof}

\subsection{Domination in o-minimal theories}
Here we recall some facts regarding o-minimal theories and the behaviour of domination in them. We refer the reader to~\cite{dominomin} for details.

\begin{fact}\label{fact:Rstabemb}
  Let $T$ be o-minimal, and $R$ a unary definable set. Then $R$ is stably embedded.
\end{fact}
\begin{proof}
  This is an easy consequence of~\cite[Lemma~1.2]{betgrrng}, as noted in~\cite[after Lemma~2.3]{pest-tri}.
\end{proof}

\begin{fact}[{\cite[Lemmas~1.24 and~1.25]{dominomin}}]\label{fct:wort1tp}
  If $T$ is o-minimal and $p,q\in S_1(A)$ are non-realised, then $p\nwort q$ if and only if there is an $A$-definable bijection $f$ such that $f_*p=q$. If $p,q\in\invtypes_1(U)$, then this is also equivalent to $p\domeq q$.
\end{fact}
By \Cref{fact-wortdomeq}, $\perp$ is preserved by $\otimes$. Under \emph{distality}, so is $\wort$. Recall that every o-minimal theory is distal, and every distal theory is $\mathsf{NIP}$, see~\cite{simon_guide_2015}.
\begin{fact}[{\cite[Lemma~1.29]{dominomin}}]\label{fact:wortotimes}
Assume $T$ is distal and $q_0,q_1, p\in\invtypes(U)$. If $q_0\wort p$  and $q_1\wort p$, then $q_0\otimes q_1\wort p$. In particular, if $p\wort q$ then for every $n,m\in \omega$ we have $p^{n}\wort  q^{m}$.
\end{fact}
\begin{fact}[\!\!{\cite[Corollary~4.7]{siminvnip}}]\label{fact:wortsimon}
  If $T$ is $\mathsf{NIP}$ and $\set{p_i: i\in I}$ is a family of pairwise weakly orthogonal invariant types then $\bigcup_{i\in I} p_i(x^i)$ is complete.
\end{fact}

\begin{fact}\label{fact:idemp}
  If $T$ is o-minimal and $p\in \invtypes_1(U)$, then for every $n>0$ we have  $p(x)\domeq  p(y_0)\otimes\ldots\otimes p(y_n)$. Moreover, this can be witnessed by a small type implying $x=y_n$.
\end{fact}
\begin{proof}
The first statement follows from \Cref{fact-wortdomeq}(\ref{point:fwd2}) and~\cite[Corollary~2.4]{dominomin}, while the `moreover' part is clear from the proof of the latter.
\end{proof}
\begin{fact}\label{fact:redto1tps}
  Assume $T$ is o-minimal and let $D(x,y)$ be the relation of being disjoint.
  \begin{enumerate}
  \item If every $p\in\invtypes(U)$ is domination-equivalent to a product of $1$-types, then $\otimes$ and $\doms$ are compatible and, if $X(\monster)$ is the set of non-realised invariant $1$-types modulo being in definable bijection, then
    \[
      (\invtilde, \otimes, \doms, \wort)\cong (\mathscr P_{<\omega}(X(\monster)), \cup, \supseteq, D).
    \]
      \item\label{point:reddeftps} If every $p\in\deftypes(U)$ is domination-equivalent to a product of $1$-types, then $\otimes$ and $\doms$ are compatible on definable types and, if $Y(\monster)$ is the set of non-realised definable $1$-types modulo being in definable bijection, then
    \[
      (\deftilde, \otimes, \doms, \wort)\cong (\mathscr P_{<\omega}(Y(\monster)), \cup, \supseteq, D).
    \]
  \end{enumerate}
\end{fact}
\begin{proof}
  The first part is~\cite[Theorem~2.13]{dominomin}. The second part is obtained by observing that the proof of the cited result still goes through when all types considered are definable.
\end{proof}
\subsection{Extremal witnesses}

In order to simplify some of the technical statements to follow, we introduce the following notion.

\begin{defin}
  Let $p(x),q(y)\in \invtypes(U)$. We call a witness $r(x,y)$ of $p\doms q$ \emph{extremal} if
  \begin{enumerate}
  \item there is a definable function $\tau$ such that $r(x,y)\proves y=\tau(x)$, and
\item if $(a,b)\models p(x)\cup r(x,y)$, then there is no $c\in \dcl(U a)\setminus U$ such that $\tp(c/U)$ is realised in $\dcl(Ub)$ and $\tp(bc/U)=\tp(b/U)\otimes \tp(c/U)$. In this case, we say that $b$ is \emph{extremal} with respect to $a$ and $U$.
\end{enumerate}
If the same $r$ also witnesses $q\doms p$, then we say that $r$ is an \emph{extremal witness of $p\domeq q$}.
\end{defin}
In other words, in the notation above, $b$ is extremal if it is not extendable `invariantly on the right' in $\dcl(U a)\setminus U$ by points with the same type as something in $\dcl(Ub)$. That is, there is no $q'$ realised in $\dcl(Ub)$ such that $q(y)\otimes q'(z)$ is realised in $\dcl(U a)\setminus U$ by some tuple of the form $(b,c)$. This does not prevent tuples of the form $(c,b)$ from realising an invariant type of the form $q'(z)\otimes q(y)$, as clear from the example below.
\begin{example}
Let $T$ be an o-minimal theory, $p_{+\infty}$ the type at infinity,  $p(x)\coloneqq p_{+\infty}(x_0)\otimes p_{+\infty}(x_1)$, and $q(y)=p_{+\infty}(y)$. If a witness of $p\doms q$ contains $x_1=y$, then it is extremal, but if it contains $x_0=y$ then it is not.
\end{example}

 The following remark is easy to prove, arguing as in the
proof of \cite[Lemma~1.24]{dominomin}.
\begin{remark}\label{remark:1dir} In an o-minimal theory, a $1$-type
$q$ is dominated by some $p$ if and only if, for every $a\models
p$, the $1$-type $q$ is realised in $\dcl(U a)$.
\end{remark}
\begin{lemma}\label{lemma:1tp1at}
Let $T$ be o-minimal, and for $i\le n$ let $q_i\in \invtypes_1(U)$ be non-realised, pairwise weakly orthogonal, and such that $q\doms q_i$.
\begin{enumerate}
\item\label{rem:extrwit} There is an extremal witness of  $q(x)\doms (\bla q0\otimes n)(y)$.
\item\label{lemma:extrcompr} For every extremal witness $r(x,y)$ of $q(x)\doms (\bla q0\otimes n)(y)$, and every $\bla p0,m\in \invtypes_1(U)$ such that every $p_i$ is domination-equivalent to some $q_j$, there is an extension $r'(x,y,z)$ of $r$ witnessing $q(x)\doms(\bla p{0}\otimes{m})(z)\otimes (\bla q0,n)(y)$. Moreover, the analogous statement where every $\doms$ is replaced by $\domeq$ also holds.
\end{enumerate}

\end{lemma}
\begin{proof} To prove Point~\ref{rem:extrwit}, fix $a\models q$. For $i\le n$, if $q_i^k$ is realised in $\dcl(Ua)$, then $k\le \dim(a/U)$. Moreover, $k\ge 1$ by \Cref{remark:1dir}. Hence, by taking a maximal such $k$ we find $b_i\models q_i$ in $\dcl(Ua)$ such that there is no $c\in \dcl(Ua)$ with $(b_i,c)\models q_i^2$. We repeat this for every $i\le n$, let $b\coloneqq (\bla b0,n)$, and observe that the type $r$ of $(a,b)$ over a suitable small set witnesses $q\doms \bla q0\otimes n$ by \Cref{fact:wortsimon}. If there was $c$ witnessing that $b$ is not extremal with respect to $U$ and $a$, then by \Cref{remark:1dir}, \Cref{fact:wortotimes}, and \Cref{fct:wort1tp} we would find $i\le n$ and a definable bijection such that $f(c)\models q_i$. We conclude by observing that $(b_i,f(c))\models q_i^2$, against the choice of $b_i$.

  For point~\ref{lemma:extrcompr}, recall that by \Cref{fct:wort1tp}, in an o-minimal theory, non-realised 1-types are either weakly orthogonal or in definable bijection. By~\cite[Corollary~1.24]{invbartheory} domination witnessed by definable functions is always compatible with $\otimes$. Using this, \Cref{fact:wortotimes}, and \Cref{fact-wortdomeq}(\ref{point:fwd5}), we find an extremal witness of
  \[
    \bla p0\otimes m\otimes \bla q0\otimes n\domeq q_{0}^{\ell_0}\otimes\ldots\otimes q_{n}^{\ell_n},
  \]
  for suitable $\ell_i$.
  \Cref{fact:idemp} tells us that for each $j$ there is an extremal witness of  $q_{i}^{\ell_i}\domeq q_i$, so we find an extremal witness $r''$ of
  \[
    \bla p0\otimes m\otimes \bla q0\otimes n\domeq \bla q0\otimes n.
  \]
From $r$ and $r''$ we easily construct the required $r'$.
\end{proof}

\begin{defin}
  Let $p(x)\in \invtypes(U)$, $A\subseteq U$, and let $\mathcal F_A^{p,1}$ be the set of $A$-definable functions with domain a set on which $p$ concentrates and codomain $U$. We define
  \[
    \pi_{p,A}(x)\coloneqq\bigcup_{f\in \mathcal F_A^{p,1}}\set{\phi(f(x))\in \mathcal L(U): p(x)\proves \phi(f(x))}.
  \]
\end{defin}

\begin{prop}\label{fact:1dimim}
Let $T$ be o-minimal and $p(x)\in \invtypes(U)$. Let $q(y)$ be any product of non-realised, invariant, pairwise weakly orthogonal $1$-types  maximal amongst those  realised in $\dcl(U a)$, for $a\models p$. For every small $A\subseteq U$ there is an extremal witness of $p\doms q$ such that $q(y)\cup r(x,y)\proves \pi_{p, A}(x)$.
\end{prop}
\begin{proof}
  Write $q=\bla q0\otimes k$.
By \Cref{lemma:1tp1at}(\ref{rem:extrwit}), there is an extremal witness $r'(x,y)$ of $p\doms q$, say containing the formula $y=\tau(x)$. Let $q'(z)=\bla q{i_0}\otimes{i_n}$ be such that $q'(z)\otimes q(y)$ is maximal amongst the products of non-realised invariant $1$-types realised in $\dcl(\monster a)$. Such a $q'$ exists by maximality of $q$ and \Cref{fct:wort1tp}.

By \cite[Proposition~2.18]{dominomin} and extremality,  $r'$ extends to some small type $r''(x,y,z)$ such that $(q'(z)\otimes q(y))\cup r''(x,y,z)\proves \pi_{p,A}(x)$. 

 We apply \Cref{lemma:1tp1at}(\ref{lemma:extrcompr}), and by discarding the coordinates in $z$ by quantifying over them as in~\cite[Lemma~2.1.10]{mennuniInvariantTypesModel2020a}, we obtain $r''(x,y)$ such that $q(y)\cup r(x,y) \proves \exists z\;((q'(z)\otimes q(y))\cup r''(x,y,z))$. It follows that $q(y)\cup r(x,y)\proves \pi_{p, A}(x)$, as desired. Moreover, by construction $r(x,y)$ still contains $y=\tau(x)$, hence is still an extremal witness of $p\doms q$.\end{proof}

\subsection{Linear o-minimal theories}
\begin{thm}\label{lineardec}
  Let $T$ be a linear expansion of $\mathsf{DOAG}$.
    Then in $T$ every invariant type is domination-equivalent to a product of $1$-types. More precisely, let $p\in \invtypes(U)$, $a\models p$,  and $\bla p0,k\in \invtypes_1(U)$ any maximal tuple of pairwise weakly orthogonal, non-realised invariant $1$-types realised in $\dcl(Ua)$. Then there is an extremal witness of $p\domeq \bla p0\otimes k$.
\end{thm}
\begin{proof}
  By~\cite[Corollary~6.3]{lp} $T$ has quantifier elimination in the language $\mathcal L$ consisting of the ordered group language, constants for all $\emptyset$-definable   points, and function symbols for all partial endomorphisms, set equal to $0$ outside their domain. 

It suffices to deal with the case of a global $A_0$-invariant type $p(x)$ of a $U$-independent tuple. Let $\mathcal N\prec \monster$ be $\abs{A_0}^+$-saturated. Apply \Cref{fact:1dimim} with $A=N$, and work in the notation of the latter. We show that $\pi_{p,N}(x)\proves p(x)$.

By quantifier elimination in $\mathcal L$, and the fact that $p$ is the type of a $U$-independent tuple, this amounts to showing that, whenever $t(x,w)$ is an $\mathcal L$-term, $d$ is a tuple from $U$, and $p(x)\proves t(x,d)>0$, then $\pi_{p,N}(x)\proves t(x,d)>0$. We show the (seemingly) stronger statement that $\pi_{p,N}(x)$ decides the cut in $U$ of $t(x,d)$.

  To this end we prove, by induction on $t(x,w)$, that if $\tilde d\in N$ is such that $\tilde d\equiv_{A_0} d$ then there is $e\in U$ such that $\pi_{p,N}(x)\proves t(x,d)=t(x,\tilde d)+e$. This suffices because, since $\tilde d\in N$, by definition $\pi_{p,N}(x)$ decides the cut of $t(x, \tilde d)$ in $U$.

 When $t(x,w)$ is a single variable, or a constant, the conclusion is clear since, as usual, we are working modulo the elementary diagram of $\monster$.  If $t(x,y) = t_0 (x,w)+t_1(x,w)$, assume that, for $i<2$, the conclusion holds for $t_i(x,w)$, witnessed by $e_i\in U$ such that
  \[
    \pi_{p,N}(x)\proves t_i(x,d)-t_i(x,\tilde d)=e_i.
  \]
  Then
  \[
    \pi_{p,N}(x)\proves t(x,d)-t(x,\tilde d)=t_0(x,d)-t_0(x,\tilde d)+t_1(x,d)-t_1(x,\tilde d)=e_0+e_1,
  \]
  and the conclusion follows with $e\coloneqq e_0+e_1$. The case where $t(x,w)=-t_0(x,w)$ is also easy, so we are left to deal with the case $t(x,w)=f(t_0(x,w))$, with $f$ a partial endomorphism.  Inductively, there is $e\in U$ such that $\pi_{p,N}(x)\proves t_0(x,d)=t_0(x,\tilde d)+e$.  Therefore, $\pi_{p,N}(x)$ decides the type of $t_0(x,d)$, and in particular it decides whether it is in $\dom f$. Since $\tilde d\equiv_{A_0} d$, and $p$ is $A_0$-invariant, we must have $p\proves t_0(x,d)\in \dom f\iff p\proves t_0(x,\tilde d)\in \dom f$.  We have two cases.
  \begin{enumerate}
  \item $\pi_{p,N}(x)\proves t_0(x,\tilde d)\notin \dom f$ (hence also $\pi_{p,N}(x)\proves t_0(x,d)\notin \dom f$). Then $f(t_0(x,d))-f(t_0(x,\tilde d))=0-0=0\in U$
  \item $\pi_{p,N}(x)\proves t_0(x,\tilde d)\in \dom f$  (hence also $\pi_{p,N}(x)\proves t_0(x, d)\in \dom f$). By invariance, $t_0(x,\tilde d)$ and $t_0(x,d)$ have the same sign, and it follows that $\pi_{p,N}(x)\proves\abs e=\abs{t_0(x,d)-t_0(x,\tilde d)}\le \max\set{\abs{t_0(x,d)}, \abs{t_0(x,\tilde d)}}$. Since $\dom f$ is a symmetric interval centred at $0$, this proves $e\in \dom f$, and the conclusion follows easily by writing $f(t_0(x,d))=f(t_0(x,\tilde d))+f(e)$.\qedhere
  \end{enumerate}
\end{proof}
\begin{cor}\label{cor:linearinvtilde}
  If $T$ is an o-minimal, linear expansion of $\mathsf{DOAG}$, then $\invtilde$ is a monoid and isomorphic to $\mathscr P_{<\omega}(X(\monster), \cup, \supseteq)$, where $X(\monster)$ is the set of non-realised invariant $1$-types modulo definable bijection.
\end{cor}
\begin{proof}
By  \Cref{lineardec} and \Cref{fact:redto1tps,fct:wort1tp}.
\end{proof}

\section{Cofinal curves and definable types}\label{sec:cofcurvesdeftypes}

Throughout this section, we fix an o-minimal $\mathcal L$-theory $T$ expanding the theory of dense linear orders without endpoints, and a monster model $\monster$. We refer indistinctively to collections of formulas and the sets that they define in $\monster$. We use $p|^+_\phi$ for the restriction of a type $p(x)$ to `positive instances' of a formula $\phi(x,y)$, that is, $p|^+_\phi$ denotes the family of all formulas of the form $\phi(x,b)$, with $b\in U^{\abs y}$, in $p(x)$.

Recall that a \emph{preorder} is a transitive and reflexive relation. We call a preordered set $(W,\preceq)$ \emph{definable} if the preorder $\preceq$ is definable (in which case, by reflexivity, $W$ is definable too). A preordered set $(W,\preceq)$ is \emph{downward directed} if, for every $x,y\in W$, there exists $z\in W$ such that $z \preceq \{x, y\}$; equivalently, for every finite subset $W_0\subseteq W$, there is $z\in W$ with $z \preceq W_0$. A family of sets (respectively, formulas) is downward directed if it is downward directed with respect to the inclusion (respectively, implication) preorder.    

\begin{fact}[\cite{guerrero_compactness}, Proposition 4.1]\label{fact:refinement}
Let $p\in S_x(U)$. For every formula $\phi(x,y)$ there exists another formula $\psi(x,z)$ such that $p|^+_\psi$ is downward directed and $p|^+_\psi \vdash p|^+_\phi$.
\end{fact}

\begin{fact}[\cite{andujarguerreroDirectedSetsTopological2021}, Corollary 25]\label{fact:cof-curves}
Suppose that $T$ expands \rcf. Let $(W,\preceq)$ be a definable downward directed preordered set. Then there exists a definable curve $\gamma:(0,\infty)\rightarrow W$ satisfying that, for every $b\in W$, there exists some $c \in (0,\infty)$ such that, for every $d \in (0,c)$, we have  $\gamma(d) \preceq b$. 
\end{fact}
Following terminology from~\cite{andujarguerreroDirectedSetsTopological2021}, we call a definable curve $\gamma$ as above a \emph{cofinal curve} in $(W,\preceq)$. 
\begin{remark}\label{rem:cofinal-parameters}
In \Cref{fact:cof-curves}, if $(W,\preceq)$ is $\mathcal L(A)$-definable, for some $A\subseteq U$, then $\gamma$ can be chosen $A$-definable too.
Indeed, suppose $\phi(x,y,d)$ defines the graph of a cofinal curve $\gamma(x)=y$ in $(W,\preceq)$. Then the set
\[S=\{t: \phi(x,y,t) \text{ defines a cofinal curve in }(W,\preceq)\}\]
is non-empty and $A$-definable. By definable choice, pick $a\in S\cap \dcl(A).$ Then $\phi(x,y,a)$ defines the desired $A$-definable cofinal curve.
\end{remark}
\begin{remark}\label{rem:inclusion-preorder}
Let $\mathcal{X}=\{X_a : a\in W\}$ be a family of sets. Consider the relation $\preceq$ on $W$ given by $a \preceq b$ whenever $X_a \subseteq X_b$. This is clearly a preorder on $W$. Observe that, if $\mathcal{X}$ is definable, then $\preceq$ is definable (over the same parameters). Clearly $(W,\preceq)$ is a downward directed set if and only if $\mathcal{X}$ is a downward directed family of sets.
\end{remark}

\begin{lemma}\label{lemma:dd1parax}
Suppose that $T$ expands \rcf. Let $A\subseteq U$ and $p\in \deftypes_x(U)$ be an $A$-definable type. For every $\mathcal L$-formula $\phi(x,y)$ there exists an $\mathcal L(A)$-formula $\psi(x,t)$, with $\abs{t}=1$, satisfying that
\begin{enumerate}
    \item\label{point:dd1parax1} $p|^+_\psi=\{\psi(x,c) : c>0\}$,
    \item\label{point:dd1parax2} for every $\phi(x,b)\in p|^+_\phi$ there exists some $c>0$ such that for every $d\in (0,c)$ we have  $\psi(x,d)\vdash \phi(x,b)$. 
\end{enumerate}
In particular $p|^+_\psi \vdash p|^+_\phi$. 
\end{lemma}
\begin{proof}
By \Cref{fact:refinement} we may assume that $\phi(x,y)$ satisfies that $p|^+_\phi$ is downward directed. Let $W=\{ b \in U^{\abs{y}} : \phi(x,b) \in p\}$. Let $\preceq$ be the directed preorder on $W$ described in \Cref{rem:inclusion-preorder}. Since $p$ is $A$-definable then $(W,\preceq)$ is $A$-definable too. By \Cref{fact:cof-curves} and \Cref{rem:cofinal-parameters} let $\gamma$ be an $A$-definable cofinal curve in $(W,\preceq)$. Consider the $\mathcal L(A)$-formula $\psi(x,t)$ given by 
\[
\exists y\; (\phi(x,y) \wedge (y = \gamma(t))).
\]
Observe that $p|^+_\psi = \{ \psi(x,c) : c>0 \}$. Since $\gamma$ is cofinal in $(W,\preceq)$, it is easy to check that condition~\ref{point:dd1parax2} is also satisfied.
\end{proof}

Denote by $p_{0^+}$ the definable type of positive infinitesimals. It is a known fact that, in an o-minimal structure $\monster$, every non-realised $A$-definable type in $\deftypes_1(U)$ is either the type at $+\infty$, the type at $-\infty$, or the type of some element infinitesimally close to some $a\in \dcl(A)$. Using the field operations one easily reaches the following fact. 

\begin{fact}\label{fact:def1tpsrcf}  In every o-minimal expansion of \rcf{} every $A$-definable $1$-type $q$ is either realised or in $A$-definable bijection with $p_{0^+}$. In particular, in the latter case $q \domeq p_{0^+}$, witnessed by the aforementioned bijection.
\end{fact}
\begin{cor}\label{cor:rcfdeftps1dec}
  Suppose that $T$ expands \rcf.  Then every non-realised definable type is domination-equivalent to the type $p_{0^+}$ of positive infinitesimals. Moreover, if $p$ is $A$-definable, then this is witnessed by a type over $A$.
  \end{cor}
\begin{proof}
Let $p \in \deftypes_x(U)$ be a non-realised $A$-definable type. As $p$ is non-realised, there is a projection $q(x_i)$ of $p(x)$ to some coordinate that is non-realised. Since $q$ is a projection, it is clearly $A$-definable and $p(x) \proves q(x_i)$. Since $q$ is $A$-definable, by \Cref{fact:def1tpsrcf} we have that $q(x_i) \domeq p_{0^+}(y)$, witnessed by any $r\in S_{pq}(A)$ placing $x_i$ and $y$ in definable bijection. We conclude that $p \cup r \vdash p_{0^+}$.

We now show that $p_{0^+} \doms p$, witnessed by a type over $A$. Let $r(x,t)$ be the collection of all $\mathcal L(A)$-formulas $\psi(x,t)$ given by \Cref{lemma:dd1parax} as $\phi(x,y)$ ranges among all $\mathcal L$-formulas $\phi(x,y)$.  By \Cref{lemma:dd1parax}(\ref{point:dd1parax1}), it is easy to see that  $p_{0^+}(t) \cup r(x,t)$ is consistent. 
Finally, by \Cref{lemma:dd1parax}(\ref{point:dd1parax2}), observe that $p_{0^+}(t) \cup r(x,t) \vdash p(x)$.
\end{proof}

\begin{thm}\label{thm:deftildeexprcf}
Let $T$ be an o-minimal expansion of $\mathsf{RCF}$. Then all non-realised definable types are domination-equivalent, and  $\deftilde$ is a monoid and isomorphic to $(\mathscr P(\set 0), \cup, \supseteq)$.
\end{thm}
\begin{proof}
This follows from \Cref{fact:redto1tps}(\ref{point:reddeftps}), \Cref{cor:rcfdeftps1dec}, and \Cref{fact:def1tpsrcf}.
\end{proof}

\section{Orthogonality in semi-bounded structures}\label{sec:semibdd}
\subsection{Preliminaries}
In this section, we fix a semi-bounded o-minimal expansion $\cal{M}=\langle M, <, +, 0, \ldots\rangle$  of an ordered group that is not linear. For convenience, we assume that the language $\mathcal L$ of $\mathcal M$ contains a constant symbol $1$ for a positive element such that the induced structure $\mathcal R$ on $R=(0,1)$ expands a real closed field whose order agrees with $<$.\footnote{\label{fn:kobi2}If instead of $1$ we named any non-zero constant, the set $R$ would still be $\emptyset$-definable by definable choice.} Note that, as far as weak orthogonality and domination are concerned, naming finitely many constants is harmless.

For the sake of brevity, until the end of the section we will abuse the notation by writing $R$ in place of $R(\mathcal M)$.
In this section, \emph{definable} means `definable in \cal M with parameters' unless otherwise specified.    By $\Lambda$ we denote the set of all $\es$-definable partial endomorphisms of $\langle M, <, +, 0\rangle$. The \emph{linear reduct} of $\mathcal M$ is ${\cal M}_\mathrm{lin}\coloneqq\langle M, <, +, 0, \{\lam\}_{\lam\in\Lambda} \rangle$ of $\cal{M}$, and its language is denoted by ${\cal L}_\mathrm{lin}$.

Following \cite{pet-sbd}, an interval $I\sub M$ is called \emph{short} if there is a definable bijection between $I$ and $R$; otherwise, it is called \emph{long}. 
Equivalently, $I$ is short if there is a definable real closed field with domain $I$ (\cite[Corollary 3.3]{pet-sbd}) whose order agrees with that of $\mathcal M$. An element $a\in M$ is called \emph{short} if either $a=0$ or $(0, \abs{a})$ is a short interval; otherwise, it is called \emph{tall}.  A tuple $a\in M^n$ is called \emph{short} if $\abs{a}\coloneqq \abs{a_1}+\ldots+\abs{a_n}$ is short, and \emph{tall} otherwise. A definable set $X\sub M^n$ (or its defining formula) is called \emph{short} if it is in definable bijection with a subset of $R^n$; otherwise, it is called \emph{long}. Notice that this is compatible, for $n=1$, with the notion of a short interval.

We  recall some information from \cite{el-sbdgroups}. Define the following closure operator. For  $A\sub M$, let
\[\scl(A)=\{a\in M : \text{there is an $A$-definable short interval that contains $a$}\}.\]

\begin{remark}
Contrary to what happens with the usual $\dcl$ operator, $\scl(A)$ grows if we pass to an elementary extension. In this section, when we write $\scl$, we mean $\scl$ computed in $\mathcal M$.
\end{remark}

By \cite[Lemma 5.5]{el-sbdgroups}, $\scl$ defines a pregeometry. Let $\rk(b/A)$ denote the usual $\scl$-dimension of a tuple $b$ over $A$.
Let $\monster\succ \mathcal M$ be a monster model, and let
\[\ldim (X)=\max\{\rk(b/A): b\in X(\monster), \text{$A$ small, and $X$ is $A$-definable}\}\] be the associated notion of dimension for a definable set $X$, which is shown in~\cite[Lemma~5.7]{el-sbdgroups} to be well-defined.\footnote{In~\cite{el-sbdgroups} $\mathrm{ldim}$ is denoted by $\mathrm{lgdim}$.}

In \cite[Corollary 5.10] {el-sbdgroups}, $\ldim$ was shown to equal a `structural dimension', which we do not define here, but notice that using that equality, we obtain that for a definable set $X$, we have $\ldim (X)=0$ if and only if $X$ is short. Moreover, the following hold:

\begin{fact}\label{fact1} For definable sets $X,Y$ and a definable function $f:X\to M^n$, we have:
\begin{enumerate}
  \item $\ldim (X\times Y)=\ldim X + \ldim Y$,
  \item\label{point:fact12} $\ldim f(X)\le \ldim X$, with equality if $f$ is a bijection.
\end{enumerate}
\end{fact}
\begin{proof}
By \cite[Lemma 4.2 and Corollary 3.11]{el-sbdgroups}, respectively.
\end{proof}

There is also a dimension formula for families, but we will not need it. Also, notably, $\ldim$ is not a definable notion.

\begin{fact}\label{fact-kobi}
Let $\mathcal M$ be any o-minimal expansion of an ordered group with a non-zero constant. Suppose $X, Y$ are two $A$-definable sets such that there is a definable bijection between them. Then there is an $A$-definable bijection between them.
\end{fact}
\begin{proof} This is similar to the argument in \Cref{rem:cofinal-parameters}.
\end{proof}

\begin{fact}\label{fact-scl}
The following are equivalent, for $a\in M$ and $A\sub M$:
\begin{enumerate}
  \item\label{point:fact-scl1} $a\in \dcl(AR)$,
  \item\label{point:fact-scl2} there is $b\in \dcl(A)$ such that $a-b$ is short,
  \item\label{point:fact-scl3} $a\in \scl(A)$.
\end{enumerate}
\end{fact}
\begin{proof}
By  \cite[Lemma 5.4]{el-sbdgroups}, $\ref{point:fact-scl2}\Lrarr\ref{point:fact-scl3}$. For $\ref{point:fact-scl1}\Rarr\ref{point:fact-scl3}$,  let $a\in \dcl(A R)$. So there is an $A$-definable map $f$ with $a\in f(R^k)$, for some $k$. By \Cref{fact1}(\ref{point:fact12}), $f(R^k)$ is short, as needed. For $\ref{point:fact-scl3}\Rarr\ref{point:fact-scl1}$, let $X$ be an $A$-definable short interval that contains $a$. 
Since $X$ is short, it is in definable bijection with $R$ and, by  \Cref{fact-kobi}, that bijection can be chosen to be $A$-definable.  Therefore, $a\in \dcl(AR)$, as needed.
\end{proof}

\begin{lemma}\label{factcY} Let $c\in M^n$ be a tuple with $\rk(c/A)=k$. Then there is an $A$-definable set $Y$ containing $c$ with $\ldim(Y)=k$.
\end{lemma}
\begin{proof}
Let $c'$ be an  $\scl$-basis of $c$ over $A$. Say $c=(c', d)$, after perhaps permuting coordinates. By \Cref{fact-scl} there is an $A$-definable map $f:M^{k+m}\to M^{n-k}$ with $f(c', r)  = d$, for some $r\in R^m$. Let $Y$ be the graph of the map $f_{\res M^k \times R^m}$.
This is an $A$-definable set containing $c$ with $\ldim Y=k$.
\end{proof}

The induced structure $\mathcal R$ of \cal M on $R$ is obviously  o-minimal. We next recall that, in the special case when \cal M  is a reduct of a real closed field, $\cal R$ is a pure real closed field, that is, a real closed field with no extra structure. Such a case arises when, for example, we start with a real closed field $\cal N=\la M, <, +, \cdot\ra$, and let \cal R be a definable (in \cal N) real closed field with bounded domain, and $\cal M=\la M, <,+, \cal R\ra$.

\begin{fact}[{\cite[Corollary 3.6]{mpp}}]\label{fact-Rind} Let $\cal M$ be a reduct of a pure real closed field \cal N. Then $\cal R$ is a pure real closed field.
\end{fact}

\begin{proof}$ $
Let $\mathcal R'$ be the reduct of $\mathcal R$ to the ordered field structure.
By \cite{opp}, there is an isomorphism $\sigma: \cal N\to \cal R'$  definable in \cal N. Now let $X\sub R^n$ be definable in $\mathcal R$. So $X$ is also definable in \cal N. Hence, $\sigma^{-1}(X)$ is definable in \cal N. But since $\sigma$ is an isomorphism between the structures \cal N and $\cal R'$, this means that $X$ is definable in $\cal R'$.
\end{proof}

\subsection{Long cones} In \cite{el-sbdgroups},  the notion of a `cone' from \cite{ed1} was refined as follows.

\begin{defn}
 If $v=(\lam_1, \ldots, \lam_n)
\in \Lam^n$ and $t\in M$, we denote $v t\coloneqq (\lambda_1 t, \ldots, \lambda_n t)$ and $\dom(v)\coloneqq \cap_{i=1}^n \dom(\lambda_i)$. We say that $v_1,\ldots, v_k\in \Lambda^n$ are \emph{independent} if for all $t_1,\ldots, t_k \in M$ with $t_i\in \dom(v_i)$,
\[
v_1 t_1+\ldots+v_k t_k=0 \,\text{ implies }\, t_1=\cdots =t_k=0.
\]
\end{defn}

\begin{defn}\label{longcone}
Let $k\in \bb{N}$. A \emph{$k$-long cone} $C\subseteq M^n$ is a definable set of the form
\[
\left\{ b+\sum_{i=1}^k v_i t_i :b\in B ,\, t_i \in J_i \right\},
\]
where $B\subseteq M^n$ is a short cell, $v_1,\ldots, v_k\in \Lambda^n$
are independent and $J_1, \ldots, J_k$ are long intervals each of the form $(0, a_i)$, for $a_i\in
M^{>0}\cup\{\infty\}$, with $J_i\sub \dom(v_i)$. So a $0$-long cone is just a short cell. A \emph{long cone} is a $k$-long cone, for some $k\in \bb{N}$. We say
that the long cone $C$ is \emph{normalised} if for each $x\in C$ there are unique $b\in B$ and $t_1\in J_1,\ldots, t_k\in J_k$ such that
$x=b+\sum_{i=1}^kv_i t_i$. In this case, we write:
\[
C=B+\sum_{i=1}^kv_i {t_i} {|J_i}.
\]
In what follows, all long cones are assumed to be normalised, and we thus drop the word `normalised'.  We will also often omit `long' from the terminology.
By a \emph{subcone of $C$} we simply mean a  cone contained in $C$.
\end{defn}
We recall the first part of the structure theorem from \cite{el-sbdgroups}. (The second part is about definable functions, but we will not use that here.)
\begin{fact}[Structure Theorem, {\cite[Theorem~3.8]{el-sbdgroups}}]\label{strtheorem} Let $X\sub M^n$ be an $A$-definable set. Then $X$ is a finite union of $A$-definable  cones.
\end{fact}
\begin{lemma}\label{fact-Adefcone}
Let $C$ be a cone as above. If $C$ is $A$-definable, then so are all of $B$ and $J_1, \ldots, J_k$.
\end{lemma}
\begin{proof}
  By definable choice, there is an element $a\in C$ which is in $\dcl(A)$. Recall that $J_i=(0, a_i)$.  It is then not hard to check that each
\[a_i=\sup \{t-s: t,s\in M, a+v_i t, a+v_i s\in C\}.\]
So every $J_i$ is $A$-definable. Moreover, 
\[B=\set*{b\in M^n: \exists (\bla t1,k)\in \bla J1\times k\;\left( b+\sum_{i=1}^kv_i t_i\in C\right)}\]
is also $A$-definable.
\end{proof}

\noindent\emph{Notation.} If $J=(0, a)$, we denote $\pm J\coloneqq (-a, a)$. Let $C=B+\sum_{i=1}^m v_i t_i|J_i$ be an $m$-long cone. We set
\[
\la C\ra \coloneqq \left\{ \sum_{i=1}^m v_i t_i : t_i \in \pm J_i\right\}.
\]
\begin{fact}[Lemma on Subcones {\cite[Lemma 3.1]{el-sbdgroups}}]\label{lemsubcones}
If $C^\prime=B^\prime+\sum_{i=1}^{m^\prime} w_i t_i | J^\prime_i$  and  $C=B+\sum_{i=1}^m v_i t_i | J_i$ are two long cones such that
$C^\prime\sub C\sub M^n$, then
$\la C^\prime\ra\sub \la C\ra$ (and hence $m^\prime\le m$).
\end{fact}

It follows that the long dimension of a $k$-cone is $k$ (\cite[Lemma 3.6(iii)]{el-sbdgroups}).

\begin{remark}\label{rem-dimC} 
A $k$-long cone $C=B+\sum_{i=1}^kv_i {t_i} {|J_i}$ has dimension $k$ if and only if $B$ is finite. In fact,
$\dim(C)=\dim(B)+ k$.
\end{remark}

\begin{lemma}\label{lem:longbox}
    Let $C\sub M^n$ be an $n$-cone. Then both $C$ and $\la C\ra$ contain boxes of the form \[B=I_1\times\ldots\times I_n,\] where each $I_i$ is a long interval and, in the case of $\la C\ra$, of the form $(-\kappa_i,\kappa_i)$.
\end{lemma}
\begin{proof}
Let $e_i\in \Lam^n$ be the unit vector, with $\mathsf 1$ in the $i$-th coordinate and $0$ everywhere else, where $\mathsf 1$ is the identical endomorphism with full domain. By  \Cref{rem-dimC}, $\la C\ra$ is an open neighbourhood of $0$. Hence, for each $i\in \set{1, \ldots, n}$ there is a positive $s_i\in M$ such that $e_i s_i\in \la C\ra$.  By \cite[Lemma 2.16]{el-sbdgroups}, there is an $n$-cone $C'\sub C$ of the form $C'=c+\sum_{i=1}^n e_i t_i | (0, \kappa_i)$, for some tall $\kappa_i\in M$. So 
\[c+(0, \kappa_1)\times \ldots\times (0, \kappa_n)\sub C,\]
is the required box for $C$. Moreover, by the Lemma on Subcones (\Cref{lemsubcones}), $\la C'\ra \sub \la C\ra$. We conclude by observing that $\la C'\ra=(-\kappa_1, \kappa_1)\times \ldots \times (-\kappa_n, \kappa_n)$ is also an open box with each $(-\kappa_i, \kappa_i)$ long. 
\end{proof}

Observe also the following fact. Let $\pi:M^{m+n}\to M^m$ denote the projection onto the first $m$ coordinates.
Below, we abuse the notation and write $0$ for a tuple $(0, \ldots, 0)$ of suitable arity.

\begin{lemma}\label{factIC}
Let $C=B+\sum_{i=1}^n v_i {t_i} {|J_i}$ be an $n$-cone contained in $M^{m+n}$,  and $I_1, \ldots, I_n\sub M$ open intervals of the form $(-d_i,d_i)$, such that
\[\{0\}\times I_1\times \ldots\times I_n\sub \la C\ra.\]
Then $\pi(\la C\ra)=\{0\}$. In other words, for each $v_i$ there is $u_i\in \Lam^n$ such that $v_i=(0, u_i)$.
\end{lemma}
\begin{proof} Let $I=\{0\}\times I_1\times \ldots\times I_n$. Since $I\sub \la C\ra$ has dimension $n$, the set $I$ is a relatively open neighbourhood of the identity in $\la C \ra$. Hence, for every $i\in \{1, \ldots, n\}$, since $v_i$ is continuous, there is a non-zero $t_i\in M$ such that the element $c_i = v_i t_i$ belongs to $I$. This implies that, for every $j\in \{1, \ldots, m\}$, the $j$-th coordinate of $c_i$ equals $0$. Since $t_i\ne 0$, it follows that the $j$-th coordinate of $v_i$ is $0$.
\end{proof}

We will also need the following lemma.

\begin{lemma}\label{lem-Cncone} Let $C\sub M^n$ be an $n$-cone of the form
\[C=\sum_{i=1}^n v_i t_i| J_i,\]
where $J_i=(0, a_i)$.
Let $b\in M^n$ be a short tuple.
Then
\[\ldim ((C+b)\sm C)<n.\]
\end{lemma}
\begin{proof} We start with a claim.

  \begin{claim*}
    There are short $b_1, \ldots, b_n\in M$, such that
    \[b=v_1 b_1 +\ldots + v_n b_n.\]
  \end{claim*}
\begin{claimproof}
Let $B=(-\kappa_1, \kappa_1)\times \ldots \times (-\kappa_n, \kappa_n)\sub \la C\ra$ be as in \Cref{lem:longbox}, with each $(-\kappa_i, \kappa_i)$ long. Since $b\in M^n$ is short, we have $b\in B\sub \la C\ra$. Therefore, there are $b_1, \ldots, b_n\in M$, such that
\[b=v_1 b_1 +\ldots + v_n b_n.\]  
By \cite[Lemma 2.6]{el-sbdgroups}, $b_1, \ldots, b_n$ are all short.
\end{claimproof}

For ease of notation, for $m\in M$, let us convene that $\infty+m=\infty$, and that $[\infty, \infty)$ denotes the empty set. Fix $\bla b1,n$ as in the Claim. Then $C+b$ is contained in the set
\[T=\left\{\sum_{i=1}^n v_i t_i: \forall i\in \set{1,\ldots, n}\; t_i\in (-\abs{b_i}, a_i+\abs{b_i}) \right\}.\]
Hence, $(C+b)\sm C$ is contained in the set
\[D\coloneqq\left\{\sum_{i=1}^n v_i t_i\in T: \exists i_0\in\{1, \ldots, n\}\; t_{i_0}\not\in (0, a_{i_0}) \right\}.\]
For $H\sub \set{1,\ldots,n}$, define
\begin{align*}
  C_H=\Bigg\{\sum_{i=1}^n v_i t_i : \quad & \forall i\notin H\; t_i\in (-\abs{b_i},  a_i+\abs{b_i}),  \text{ and}   \\ & \forall i\in H \; t_i\in (-\abs{b_i}, 0]\cup [a_i, a_i +\abs{b_i})\Bigg\}.
\end{align*}
By definition, $D=\bigcup\set{C_H :\emptyset \ne H\subseteq \set{1,\ldots, n}} $.
It is not hard to check that if $H\ne \emptyset$ then $C_H$ is in definable bijection with a set of the form $Y=Y_1\times \ldots \times Y_n$, with at least one of the $Y_i$ being short and of the form
\[(-\abs{b_i}, 0]\cup [a_i, a_i+\abs{b_i}),\] via the map $\tau : Y\to M^n, \, \tau(t_1, \ldots, t_n)=(v_1 t_1, \ldots, v_n t_n)$. It follows that $\ldim Y<n$, implying that $\ldim (C+b)\sm C<n$, as needed.
\end{proof}

\begin{lemma}\label{lem-Cncone2} Let $C\sub M^n$ be an $n$-cone of the form
\[C=\sum_{i=1}^n v_i t_i| J_i,\]
where $J_i=(0, a_i)$.
Let $B\subseteq M^n$ be a short cell, and denote
\[T\coloneqq\bigcup_{b\in B} \left(b+ C\right)\qquad\text{and}\qquad S\coloneqq\bigcap_{b\in B} \left(b+C\right).\]
Then $\ldim (T\sm S)<n$.
\end{lemma}
\begin{proof} We have
\[T\sm S=\bigcup_{b\in B} \bigcup_{b'\in B} (b+C)\sm \left(b'+ C\right).\]
Hence, by a standard application of~\cite[Lemma 4.2(i)]{el-sbdgroups}, it suffices to prove that for every $b, b'\in B$,
\[\ldim \, (b+C)\sm \left(b'+ C\right)<n.\]
Subtracting $b'$ from both sets, it is equivalent to show that for $b''=b-b'$, we have
\[\ldim \, (b''+C)\sm C<n.\]
Since $b,b'$ belong to the short cell $B$, the element $b''$ is short. Hence the above inequality is given by  \Cref{lem-Cncone}.
\end{proof}

\subsection{Orthogonality and short closure}

\begin{lemma}\label{lem-scllin}
Let $c\in M^n$ be $\scl$-independent over $A$. Then its type over $A$ is determined by its type over $A$ in the linear reduct~${\cal M}_\mathrm{lin}=\langle M, <, +, 0, \{\lam\}_{\lam\in\Lambda} \rangle$. That is, if $X$ is an $A$-definable set containing $c$, then there is an $\cal L_\mathrm{lin}(A)$-definable set $Y\sub X$ containing $c$.
\end{lemma}
\begin{proof}
By \Cref{strtheorem}, we may assume that $X$ is an $A$-definable cone $X=B+\sum_{i=1}^kv_i {t_i} {|J_i}$. Since $c$ is $\scl$-independent over $A$, it must be $k\ge n$, and since $X\sub M^n$, it follows that $k=n$. Moreover, by \Cref{rem-dimC}, $\dim X=\dim B+n$. Hence $\dim B=0$, and since $B$ is a cell, it follows that $B=\{b\}$ is a singleton. Hence $X=\{b\}+\sum_{i=1}^kv_i {t_i} {|J_i}$ is definable in the linear reduct, and we conclude by \Cref{fact-Adefcone}.
\end{proof}

\begin{lemma}\label{lem-bc}
Fix $A\sub  M$. Let $b\in \ M^m$ have $\rk(b/A)=0$ and $c\in  M^n$ have $\rk(c/A)=n$.
Let $X\sub  M^{m+n}$ be an $A$-definable set containing $(b,c)$. Then there are $A$-definable sets $Y\subseteq M^m$ and $Z\subseteq M^n$ such that  $(b,c)\in Y\times Z\subseteq X$.
\end{lemma}
\begin{proof} By \Cref{factcY}, we may assume $\ldim X\le n$. Indeed, we may intersect $X$ with an $A$-definable set $X'$ that contains $(b,c)$ and has $\ldim X'=n$. Since $\rk(c/A)=n$, we actually have $\ldim X=n$. Now, by \Cref{strtheorem}, $X$ is a finite union of $A$-definable cones, and hence we may assume it is one such $n$-cone, $X=B+\sum_{i=1}^n v_i {t_i} {|J_i}$, with $J_i=(0, a_i)$. Let $\pi:M^{m+n}\to M^m$ be the projection onto the first $m$ coordinates.

  \begin{claim}\label{claim:zeroproj}
    We have $\pi(\la X\ra)=\{0\}$, that is, each $v_i$ is of the form $v_i= (0, u_i)$, for some $u_i\in \Lam^n$.
  \end{claim}
\begin{claimproof}[\it Proof of \Cref{claim:zeroproj}] By the Structure Theorem, the fiber $X_b$ of $X$ above $b$ is a finite union of $Ab$-definable cones. Since $b\in \scl(A)$ and $c$ is $\scl$-independent over $A$, it is contained in some $n$-cone $K$ among them. By \Cref{lem:longbox}, there is an open box $I=I_1\times \ldots\times I_n\sub  K$, where each $I_i$ is a long interval.\footnote{Note that even if it were $\pi(B)\times I\sub X$, we could not conclude the proof of the lemma, because $I$ need not be $A$-definable and need not contain $c$.}  The cone $T=\{b\}\times I$ is then contained in  $X$,  with $\la T\ra=\{ 0\}\times J_1\times \ldots\times J_n$, for some open intervals $J_1, \ldots, J_n\sub M $. By the Lemma on Subcones (\Cref{lemsubcones}), $\la T\ra\sub \la X\ra$. We conclude by \Cref{factIC}.
\end{claimproof}
For $i=1, \ldots, n$, let $u_i\in \Lam^n$ be given by \Cref{claim:zeroproj}. Writing $p:M^{m+n}\to M^n$ for the projection onto the last $n$ coordinates, and setting $C\coloneqq \sum_{i=1}^n u_i (0, a_i)$, we obtain:
\[X=B+\sum_{i=1}^n v_i (0,a_i)= \bigcup_{\beta \in B}\left(\{\pi(\beta )\}\times \left(p(\beta )+C\right)\right).\]

\begin{claim}\label{claim:cbigcap}
  We have $c\in \bigcap_{\beta \in B} (p(\beta )+C)$.
\end{claim}
\begin{claimproof}[\it Proof of \Cref{claim:cbigcap}]
Observe that $c\in \bigcup_{\beta \in B} (p(\beta )+C)$. Since $p(B)$ is short, by \Cref{lem-Cncone2} we obtain:
\[\ldim\left(\bigcup_{\beta \in B} (p(\beta )+C) \sm \bigcap_{\beta \in B} (p(\beta )+C)\right)<n.\]
Since the above set is $A$-definable, and $c$ is $\scl$-independent over $A$, we obtain that $c$ does not belong to it. Therefore, $c\in \bigcap_{\beta \in B} (p(\beta )+C).$
\end{claimproof}
To finish the proof of the lemma, observe that, by \Cref{claim:zeroproj,claim:cbigcap},
\[(b ,c)\in \pi(B)\times \bigcap_{\beta \in B} (p(\beta )+C)\sub X.\]
Since  $Y\coloneqq \pi(B)$ and $Z\coloneqq \bigcap_{\beta \in B} (p(\beta )+C)$ are $A$-definable by \Cref{fact-Adefcone}, we are done.
\end{proof}

\begin{prop}\label{prop-bcorth}
Let $A\sub M$, and $b,b'\in M^m, c\in M^n$ be such that $\rk(b/A)=0$, $\tp(b/A)=\tp(b'/A)$, and $\rk(c/A)=n$. Let $X\sub M^{m+n}$ be an $A$-definable set. Then $X$ contains $(b,c)$ if and only if $X$ contains $(b',c)$.
\end{prop}
\begin{proof}
By \Cref{lem-bc}, we may assume that $X$ is of the form $X=Y\times I$. Since $Y$ is an $A$-definable set containing $b$, and $\tp(b/A)=\tp(b'/A)$, we have $b'\in Y$. Hence $(b', c)\in X$.
\end{proof}

\begin{remark}
  If $\rk(c/A)<n$, then the conclusion of \Cref{prop-bcorth} fails. Indeed, let $b,b'\in M $ be distinct positive $A$-infinitesimals, that is, distinct realisations of $p_{0^+}\restr A$. Let $d\in M$ have $\rk(d/A)=1$, e.g.\ $d>\dcl(A)$, and let $c=(d, b+d)$. If $X$ is the graph of the addition map $M^2\to M$, then $(b, c)\in  X$ but $(b',c)\notin X$.
\end{remark}

\begin{defin}
  Let $p(x)\in S(A)$ and $b\models p$. We define $\rk(p)\coloneqq\rk(b/A)$. If $\rk(p)=0$ we call $p$ \emph{short}; if $\rk(p)=\abs x$, we call $p$ \emph{\scl-independent}.
\end{defin}

\begin{cor}\label{co:wort}
Let $p(x), q(y)\in S(A)$. If $p$ is short and $q$ is \scl-independent, then $p\wort q$.
\end{cor}
\begin{proof}
Immediate from \Cref{prop-bcorth}.
\end{proof}

\section{Domination in o-minimal expansions of ordered groups}\label{sec:mainresults}
We resume studying types over a monster model $\monster$. As in \Cref{sec:domination}, their realisations, as well as sets $B\supseteq U$, are assumed to be in a larger monster model $\monsterbis$ such that $B$ is small with respect to $\monsterbis$.

Until \Cref{thm:main1}, we will work in a non-linear semi-bounded structure, and use the same notation as in \Cref{sec:semibdd}, except that we replace $\mathcal M$ by $\monsterbis$. In particular, $\scl(A)$ will coincide with $\dcl(A R(\monsterbis))$.

We begin with a version of \Cref{co:wort} for orthogonality of invariant types.
\begin{lemma}\label{lemma:sclrkinvext}
  For every $p\in \invtypes(U)$ and $B\supseteq U$ we have $\rk(p)=\rk(p\invext B)$.
\end{lemma}
\begin{proof}
  It is immediate that $\rk(p\invext B)\le \rk(p)$. For the opposite inequality, fix $a=(\bla a0,m)\models p\invext B$. It suffices to show that if $a_0\in \scl(B \bla a1,m)$ then for every $c=(\bla c0,m)\models p$ we have that $c_0\in \scl(U \bla c1,m)$.  If the former holds, then there is a $B\bla a1,m$-definable short interval containing $a_0$, say defined by $\phi(x_0, \bla a1,m, b)$. By \Cref{defin:invtensorth}(\ref{point:definvext}), if $p$ is $A$-invariant there is  $\tilde b\in U$ with $\tilde b\equiv_A b$ such that $p(x)\proves \phi(x_0, \bla x1,m, \tilde b)$. Whether the interval $\phi(x_0, \bla c1,m, \tilde b)$ is short is a property of $\tp(\bla c1,m,\tilde b)$,  and by \Cref{defin:invtensorth}(\ref{point:definvext})
  \[
    \bla a1,m, b\equiv \bla a1,m,\tilde b\equiv \bla c1,m,\tilde b.
  \]
  Therefore $c_0$ belongs to the short interval defined by $\phi(x_0, \bla c1,m, \tilde b)$.
\end{proof}
\begin{prop}\label{prop-sbd}
  Let $p(x), q(y)\in \invtypes(U)$. If $p$ is short and $q$ is \scl-independent, then $p\perp q$.
\end{prop}
\begin{proof}
By \Cref{co:wort} and \Cref{lemma:sclrkinvext}.
\end{proof}

Recall that by \Cref{fact:Rstabemb} $R$ is stably embedded. 

\begin{lemma}\label{lem-splitRsclind}
 For any tuple $a\in V$ and any \scl-basis of $a$ over $U$ there is a tuple $b\sub R$ such that $\dcl(aU)=\dcl(bcU)$.
\end{lemma}
\begin{proof}
 Let $c$ be an \scl-basis of $a$ over $U$. Then $a\in \dcl(R(\monsterbis) U c)$, so there are $e\in R(\monsterbis)^n$  and an $U c$-definable map $f:V^n\to V^{\abs a}$, such that $a=f(e)$. The set
\[X=f^{-1}(a)\cap R^n=\{x\in R^n: f(x)=f(e)\}\]
is $U ec$-definable and contained in $R^n$. Hence
 $X$ is definable in $\cal R(\monsterbis )$ by stable embeddedness. Let $b\sub R(\monsterbis)$ be a canonical parameter for it in the sense of $\cal R(\monsterbis)$.   Now let $\sigma$ be an automorphism of $\monsterbis $ that fixes $U c$ pointwise. Then $\sigma\restr R(\monsterbis)$ is an automorphism of $\cal R(\monsterbis)$.
Recall that $c\subseteq a$, and observe that $X$ is also $Ua$-definable. Moreover, $\set{a}=f(X)$, hence we have
\[\sigma(a)=a\iff \sigma (X)=X\iff \sigma (b)=b,\]
showing that $\dcl(bcU )=\dcl(aU )$, as needed.
\end{proof}

As $\mathcal R$ is stably embedded, we may use \Cref{fact-fullemb} and, for ease of notation, conflate types in the sense of $\mathcal R(\monster)$ with types in the sense of $\monster$ concentrating on a cartesian power of $R$.

In order to simplify statements, in what follows we allow the degenerate case $k=-1$, in which case $p\domeq p_R$.
\begin{lemma}\label{lem:rfsbteorcf}
  Let $T$ be a semi-bounded o-minimal theory with a definable interval $R$ such that $\mathcal R$ expands a real closed field. For every  $p\in \invtypes(U)$ there are a definable function $\tau^p=(\rho^p,\lambda^p)$ and pairwise weakly orthogonal invariant types $p_R=\rho^p_*p\in \invtypes(R(\monster))$ and $\bla p0,k\in \invtypes_1(U)$ such that
\begin{enumerate}
\item\label{point:techlem1} for $0\le i\le k$, each $p_i$ is the type of an element of \scl-rank $1$; and
\item\label{point:techlem2} there is a small type containing $y=\tau^p(x)$ that is an extremal witness of $p\domeq p_R\otimes \bla p0\otimes k=\rho^p_*p\otimes \lambda^p_*p$.
  \newcounter{5.4counter}
  \setcounter{5.4counter}{\value{enumi}}
\end{enumerate}
Furthermore, if $a\models p$,
\begin{enumerate}
  \setcounter{enumi}{\value{5.4counter}}
\item\label{point:techlem3} as $\bla p0,k$ we may take any tuple of pairwise weakly orthogonal $1$-types of $\scl$-rank $1$ that is maximal amongst those realised in $\dcl(Ua)$;
\item\label{point:techlem4} the product of any such is domination-equivalent to the type $q$ of some $\scl$-basis of $a$ over $U$; and 
\item\label{point:techlem5} as $p_R$ we may take any element of $\invtypes(R(\monster))$ such that  $p_R\otimes q$ is in definable bijection with $p$; in particular, the dimension of $p_R$ equals $\dim(R(\dcl(Ua))/U)$.
\end{enumerate}
\end{lemma}
\begin{proof}
  
  Let $a\models p$. Observe immediately that, for any choice of $\bla p0,k$ as in point~\ref{point:techlem3}, and any $\bla c0,k\models p_0(x_0)\cup\ldots\cup p_k(x_k)$, we may extend $\bla c0,k$ to an \scl-basis $c$ of $a$ over $U$. Let $b\subseteq R(\dcl (Ua))$ be such that $a$ and $bc$ are in definable bijection, say via definable functions $\rho^p$ and $\lambda_0^p$ mapping $a$ to $b$ and $c$ respectively; at least one such $b$ exists by \Cref{lem-splitRsclind}. Note immediately that we are not fixing any particular such $b$, $c$,  so as to take care of points~\ref{point:techlem5} and~\ref{point:techlem3} (hence also of point~\ref{point:techlem1}).

Set $p_R\coloneqq \rho^p_*p=\tp(b/U)$ and $q\coloneqq(\lambda_0^p)_*p=\tp(c/U)$. We have  $\tp(a/U)\domeq \tp(bc/U)$, and the types $p_R$ and $q$ are invariant by~\cite[Lemma~1.8]{invbartheory}, and weakly orthogonal by \Cref{co:wort}. By \Cref{fact-wortdomeq,fact-fullemb}, it suffices to find a definable function $\lambda^p$ and a small type $r''$ containing $w=\lambda^p(z)$ and witnessing extremally $q(z)\domeq (\bla p0\otimes k)(w)$. From this, we obtain point~\ref{point:techlem2} by setting $\tau^p\coloneqq(\rho^p,\lambda^p)$, and simultaneously prove point~\ref{point:techlem4}. We will obtain the required $\lambda^p$ by composing $\lambda_0^p$ with suitable definable functions.

Let $q'(z)$ be the restriction of $q$ to the linear language. By \Cref{lem-scllin} $q'\proves q$. 
In the linear reduct of $T$, apply \Cref{lineardec} to $q$ and some  $q_0'\otimes\ldots\otimes q_n'$ satisfying its assumptions and such that, for $i\le k$, the type $p_i$ is the unique (by o-minimality) completion in $T$ of  $q_{n-k+i}'$. Hence, in the linear reduct, we have a definable function $\tau$, a small type $r'$ implying $w=\tau(z)$, and pairwise orthogonal invariant $1$-types $q_i'$ such that $r$ witnesses extremally  $q'(z)\domeq (q_0'\otimes\ldots\otimes q_n')(w)$. 

Let $\lambda^p_1\coloneqq\tau\circ\lambda_0^p$ and, for $-n+k\le i\le -1$, define $p_i$ to be the unique completion of $q_{n-k+i}'$. By construction, if $r'$ witnesses $q_0'\otimes\ldots\otimes q_n'\doms q'$, we have  $\bla p{-n+k}\otimes k\cup r'\proves q_0'\otimes\ldots\otimes q_n'\cup r'\proves q'\proves q$. 

\begin{claim*}
  For every $i<0$ there is $j\ge 0$ such that $p_i\domeq p_j$.
\end{claim*}
\begin{claimproof}
Clearly $\tau$, being defined in the linear reduct of $T$, is a fortiori definable in $T$. Now, $w=\tau(z)$, and $q(z)$ says that $z$ is \scl-independent over $U$. Since no coordinate $w_i$ of $w$ is realised, it follows from \Cref{co:wort} and \Cref{fact-wortdomeq}(\ref{point:fwd3}) that no $p_i$ is short, hence all of them must have $\scl$-rank $1$. The conclusion follows from maximality of $\bla p0,k$ and \Cref{fct:wort1tp}.
\end{claimproof}

 Compose  $\lambda_1^p$ with the projection on the last $k+1$ coordinates, obtaining a function $\lambda^p$. By the Claim, we may apply \Cref{lemma:1tp1at}(\ref{lemma:extrcompr}), and correspondingly modify $r'$ to obtain the required $r''$, hence concluding the proof.
\end{proof}
\begin{defin}
We define $\longpart(\monster)$ to be the quotient of $\set{p\in\invtypes_1(U): \rk(p)=1}$ by domination-equivalence.
\end{defin}
\begin{remark}\label{rem:onlylongdeftp}
  The only definable $1$-type $p$ with $\rk(p)=1$ is the type at $+\infty$.
\end{remark}

\begin{remark}
The monoid $(\mathscr P_{<\omega}(\operatorname{Long}(\monster)), \cup, \supseteq)$ may be  naturally viewed as a quotient of the domination monoid $\invtildeof{\monster_\mathrm{lin}}$ of the linear reduct. Namely, we identify two generators $\class p$, $\class q$,  where $p,q\in \invtypes_1(\monster_\mathrm{lin})$, whenever
\begin{enumerate}
\item $p$ and $q$ are in definable bijection in $\monster$, or
\item $p$ and $q$ are both short.
\end{enumerate}
 It is easy to see that this is well-defined and extends to a surjective homomorphism of partially ordered monoids.
\end{remark}
\begin{thm}\label{thm:main1}
    Let $T$ be a semi-bounded o-minimal theory with a definable interval $R$ such that its induced structure $\mathcal R$ expands a real closed field.
There is an isomorphism of posets
  \[
    \invtilde\cong \invtildeof{\mathcal R(\monster)}\times (\mathscr P_{<\omega}(\operatorname{Long}(\monster)), \supseteq).
  \]
If $\otimes$ is compatible with $\doms$ in $\mathcal R(\monster)$, then the same holds in $\monster$, and if $\mathscr P_{<\omega}(\longpart(\monster))$ is equipped with the binary operation $\cup$, then the above is an isomorphism of monoids.
\end{thm}
\begin{proof}
  \Cref{lem:rfsbteorcf}, together with the definition of $\longpart(\monster)$, \Cref{fct:wort1tp} and \Cref{prop-sbd}, gives us the required isomorphism of posets. Moreover, if $\invtildeof{\mathcal R(\monster)}$ and $\invtilde$ are monoids, then the above is also easily seen to be an isomorphism of monoids.

In order to show transfer of compatibility of $\otimes$ with $\doms$
we need to show that, assuming that it holds in $\mathcal R(\monster)$, if $p,q_0,q_1\in \invtypes(U)$ and $q_0\doms q_1$, then $p\otimes q_0\doms p\otimes q_1$. Apply to $p, q_0, q_1$ \Cref{lem:rfsbteorcf}, obtaining $\tau^p=(\rho^p, \lambda^p)$, and analogously for $q_0$ and $q_1$.

Let $s\in \invtypes_1(U)$. By \Cref{fact:wortotimes}, if $p\otimes q_i\doms s$ then $p\doms s$ or $q_i\doms s$. Again by \Cref{fact:wortotimes}, and by \Cref{co:wort}, if $\rk(s)=1$ then we must have $\lambda^p_*p\doms s$ or $\lambda^{q_i}_*q_i\doms s$.

For $i<2$ consider the definable function $\tau^i\coloneqq (\rho^p, \rho^{q_i}, \lambda^p, \lambda^{q_i})$. By the above paragraph and the dimension considerations in \Cref{lem:rfsbteorcf}(\ref{point:techlem5}), we see that $\tau^i$ satisfies the conclusion of  \Cref{lem:rfsbteorcf} for $p\otimes q_i$, with the possible exception of the weak orthogonality requirement; namely, the $1$-types constituting the factors of $(\lambda^p, \lambda^{q_i})_*p\otimes q_i$ contain a maximal tuple as in \Cref{lem:rfsbteorcf}(\ref{point:techlem3}), but this containment may be strict, as it may happen that some $1$-type of $\scl$-rank $1$ is dominated by both $\lambda^p_*p$ and $\lambda^{q_i}_*q_i$. Nevertheless, because $q_i\domeq\tau^{q_i}_*q_i$ is witnessed extremally, it is easy to see that there is a definable function $\pi$ consisting in a projection on a suitable set of coordinates followed by a permutation of the remaining ones such that $\pi\circ\tau^i$ satisfies the conclusion of \Cref{lem:rfsbteorcf} for $p\otimes q_i$. Applying \Cref{fact:wortotimes,,fact:wortsimon,,fact:idemp} yields that $\tau^i_*(p\otimes q_i)\domeq(\pi\circ\tau^i)_*(p\otimes q_i)$, and it follows that  $\tau^p_* p\otimes \tau^{q_i}_* q_i\domeq p\otimes q_i$.  By arguing as in the proof of \cite[Proposition~1.3]{hilsDominationMonoidHenselian2024}, it now suffices to show that if $q_0\doms q_1$ then $\tau^p_* p\otimes \tau^{q_0}_* q_0\doms \tau^p_* p\otimes \tau^{q_1}_* q_1$.

Recall that by \Cref{prop-sbd} and \Cref{fact-wortdomeq}(\ref{point:fwd7}) we have $\rho^p_*p\perp \lambda^p_*p$;  hence,

\begin{equation}
  \label{eq:taurholambda}\tag{$*$}
  \tau^p_* p\domeq \rho^p_* p\otimes\lambda^p_* p \domeq \lambda^p_* p\otimes \rho^p_*p,
\end{equation}
and analogously for the $q_i$. By \Cref{co:wort} and~\cite[Lemma~2.6]{hilsDominationMonoidHenselian2024}, if $q_0\doms q_1$ then $\rho^{q_0}_*{q_0}\doms \rho^{q_1}_*{q_1}$, and since $\otimes$ is compatible with $\doms$ in $\mathcal R$, we have $\rho^p_*p\otimes\rho^{q_0}_*{q_0}\doms \rho^p_*p\otimes \rho^{q_1}_*{q_1}$.

Since $\lambda^p_* p$ is a product of $1$-types, and similarly for the $q_i$, it follows from \Cref{lemma:1tp1at} and $q_0\doms q_1$ that $\lambda^p_* p\otimes \lambda^{q_0}_*q_0\doms\lambda^p_* p\otimes\lambda^{q_1}_* q_1$. By \Cref{fact-wortdomeq}(\ref{point:fwd7}) we have $\rho^p_*p\otimes \rho^{q_1}_*{q_1}\perp \lambda^p_* p\otimes\lambda^{q_1}_* q_1$, hence by orthogonality
\[
  \rho^p_*p\otimes\rho^{q_0}_*{q_0}\otimes \lambda^p_* p\otimes \lambda^{q_0}_*q_0\doms \rho^p_*p\otimes \rho^{q_1}_*{q_1}\otimes \lambda^p_* p\otimes\lambda^{q_1}_* q_1.
\]
Again by orthogonality, we may easily rearrange the terms in the products above and use \eqref{eq:taurholambda} to conclude that $\tau^p_* p\otimes \tau^{q_0}_* q_0\doms \tau^p_* p\otimes \tau^{q_1}_* q_1$, as promised.
\end{proof}
 
\begin{cor}\label{co:RpureRCF} Let $T$ be any theory which expands $\mathsf{DOAG}$ and is a reduct of $\mathsf{RCF}$. Then every invariant type is domination-equivalent to a product of $1$-types.
\end{cor}
\begin{proof} 
If $T$ is linear, we already proved this in \Cref{lineardec}, and if $T=\mathsf{RCF}$ this is~\cite[Proposition~5.10]{dominomin}. The remaining case is that of a semi-bounded $T$. By \Cref{fact-Rind}, $\mathcal R$   is a real closed field with no extra structure. We then conclude by \Cref{lem:rfsbteorcf}  and a further use of~\cite[Proposition~5.10]{dominomin}. 
\end{proof}

\begin{thm}\label{thm:deftildedoag}
  If $T$ is an o-minimal expansion of $\mathsf{DOAG}$, then $\deftilde$ is a monoid, and isomorphic to 
  \begin{itemize}
  \item $(\mathscr P(\set{0,1}), \cup, \supseteq)$ if $T$ is semi-bounded;
  \item $(\mathscr P(\set{0}), \cup, \supseteq)$ otherwise.
  \end{itemize}
\end{thm}
\begin{proof}
If $T$ is not semi-bounded, then by \cite{ed1} it expands $\mathsf{RCF}$, and we already dealt with this case in \Cref{thm:deftildeexprcf}.
  
If $T$ is  semi-bounded, it follows from \Cref{thm:main1}, \Cref{fact:def1tpsrcf}, \Cref{rem:onlylongdeftp} and \Cref{co:wort} that there are precisely two non-realised $1$-types up to definable bijection: any non-realised short type (for instance, the type of an infinitesimal), and the type at $+\infty$. Moreover, the combination of \Cref{lem:rfsbteorcf} with \Cref{thm:deftildeexprcf} tells us that every definable type is domination-equivalent to a product of $1$-types. We conclude by \Cref{fact:redto1tps}.
\end{proof}


\begin{thebibliography}{AGTW21}

\bibitem[AG25]{guerrero_compactness}
Pablo And\'ujar~Guerrero.
\newblock Definable compactness in o-minimal structures.
\newblock {\em Model Theory}, 2025+.
\newblock To appear, preprint available at
  \url{https://arxiv.org/abs/2405.07112}.

\bibitem[AGTW21]{andujarguerreroDirectedSetsTopological2021}
Pablo And{\'u}jar~Guerrero, Margaret E.~M. Thomas, and Erik Walsberg.
\newblock Directed sets and topological spaces definable in o-minimal
  structures.
\newblock {\em Journal of the London Mathematical Society}, 104(3):989--1010,
  2021.

\bibitem[Edm00]{ed1}
Mario~J. Edmundo.
\newblock Structure theorems for o-minimal expansions of groups.
\newblock {\em Annals of Pure and Applied Logic}, 102(1):159--181, 2000.

\bibitem[EHM19]{ehm}
Clifton Ealy, Deirdre Haskell, and Jana Maříková.
\newblock Residue {Field} {Domination} in {Real} {Closed} {Valued} {Fields}.
\newblock {\em Notre Dame Journal of Formal Logic}, 60(3):333--351, 2019.

\bibitem[Ele12]{el-sbdgroups}
Pantelis~E. Eleftheriou.
\newblock Local analysis for semi-bounded groups.
\newblock {\em Fundamenta Mathematicae}, 216(3):223--258, 2012.

\bibitem[EP12]{elp-sel2}
Pantelis~E. Eleftheriou and Ya’acov Peterzil.
\newblock Definable groups as homomorphic images of semi-linear and
  field-definable groups.
\newblock {\em Selecta Mathematica}, 18(4):905--940, 2012.

\bibitem[GY23]{gy}
Kyle Gannon and Jinhe Ye.
\newblock An {{Invitation}} to {{Extension Domination}}.
\newblock {\em Notre Dame Journal of Formal Logic}, 64(3):253--280, 2023.

\bibitem[HHM08]{hhm}
Deirdre Haskell, Ehud Hrushovski, and Dugald Macpherson.
\newblock {\em Stable {Domination} and {Independence} in {Algebraically}
  {Closed} {Valued} {Fields}}, volume~30 of {\em Lecture {Notes} in {Logic}}.
\newblock Cambridge University Press, 2008.

\bibitem[HM24]{hilsDominationMonoidHenselian2024}
Martin Hils and Rosario Mennuni.
\newblock The domination monoid in henselian valued fields.
\newblock {\em Pacific Journal of Mathematics}, 328(2):287--323, 2024.

\bibitem[LP93]{lp}
James Loveys and Ya’acov Peterzil.
\newblock Linear {O}-minimal structures.
\newblock {\em Israel Journal of Mathematics}, 81(1-2):1--30, 1993.

\bibitem[May88]{mayer_vaughts_1988}
Laura~L. Mayer.
\newblock Vaught's {Conjecture} for o-{Minimal} {Theories}.
\newblock {\em The Journal of Symbolic Logic}, 53(1):146--159, 1988.

\bibitem[Men20a]{mennuniInvariantTypesModel2020a}
Rosario Mennuni.
\newblock {\em Invariant Types in Model Theory}.
\newblock PhD thesis, University of Leeds, 2020.

\bibitem[Men20b]{invbartheory}
Rosario Mennuni.
\newblock Product of invariant types modulo domination–equivalence.
\newblock {\em Archive for Mathematical Logic}, 59:1--29, 2020.

\bibitem[Men22a]{dominomin}
Rosario Mennuni.
\newblock The domination monoid in o-minimal theories.
\newblock {\em Journal of Mathematical Logic}, 22(01):2150030, 2022.

\bibitem[Men22b]{mennuni_weakly_2022}
Rosario Mennuni.
\newblock Weakly binary expansions of dense meet-trees.
\newblock {\em Mathematical Logic Quarterly}, 68(1):32--47, 2022.

\bibitem[MPP92]{mpp}
David Marker, Ya'acov Peterzil, and Anand Pillay.
\newblock Additive {Reducts} of {Real} {Closed} {Fields}.
\newblock {\em J. Symbolic Logic}, 57:109--117, 1992.

\bibitem[MS94]{marker_definable_1994}
David Marker and Charles Steinhorn.
\newblock Definable {Types} in {O}-minimal {Theories}.
\newblock {\em Journal of Symbolic Logic}, 59:185--198, 1994.

\bibitem[OPP96]{opp}
Margarita Otero, Ya'acov Peterzil, and Anand Pillay.
\newblock On {Groups} and {Rings} {Definable} {In} {O}-{Minimal} {Expansions}
  of {Real} {Closed} {Fields}.
\newblock {\em Bulletin of the London Mathematical Society}, 28(1):7--14, 1996.

\bibitem[Pet92]{pet-reals}
Ya'acov Peterzil.
\newblock A structure theorem for semibounded sets in the reals.
\newblock {\em The Journal of Symbolic Logic}, 57(3):779--794, 1992.

\bibitem[Pet09]{pet-sbd}
Ya'Acov Peterzil.
\newblock Returning to semi-bounded sets.
\newblock {\em The Journal of Symbolic Logic}, 74(2):597--617, 2009.

\bibitem[PS98]{pest-tri}
Y~Peterzil and S~Starchenko.
\newblock A trichotomy theorem for o-minimal structures.
\newblock {\em Proceedings of the London Mathematical Society}, 77(3):481--523,
  1998.

\bibitem[PSS89]{betgrrng}
A.~Pillay, P.~Scowcroft, and C.~Steinhorn.
\newblock Between {Groups} and rings.
\newblock {\em Rocky Mountain Journal of Mathematics}, 19(3):871--886, 1989.

\bibitem[Sim14]{siminvnip}
Pierre Simon.
\newblock Invariant types in {NIP} theories.
\newblock {\em Journal of Mathematical Logic}, 15(02):1550006, 2014.

\bibitem[Sim15]{simon_guide_2015}
Pierre Simon.
\newblock {\em A guide to {NIP} theories}, volume~44 of {\em Lecture {Notes} in
  {Logic}}.
\newblock Cambridge University Press, 2015.

\end{thebibliography}
\end{document}